\documentclass{amsart}
\usepackage{amssymb}

\newcommand{\R}{\mathbf{R}}

\newtheorem{theorem}{Theorem}[section]

\newtheorem{lemma}[theorem]{Lemma}
\newtheorem{proposition}[theorem]{Proposition}
\theoremstyle{remark}
\newtheorem{remark}[theorem]{Remark}

\numberwithin{equation}{section}

\begin{document}

\title[Infinite time blow-up and slow decay]
{Infinite time blow-up and slow decay for 
the six dimensional energy-critical heat equation\\
with self-similarly decaying initial data}

\author[K. Hisa]{Kotaro Hisa}
\address[Kotaro Hisa]{Department of Applied Mathematics, 
Faculty of Science, 
Fukuoka University, 
8-19-1 Nanakuma, Jonan-ku,  Fukuoka 814-0180, Japan}
\email{hisak@fukuoka-u.ac.jp}

\author[J. Takahashi]{Jin Takahashi}
\address[Jin Takahashi]{Department of Mathematical and Computing Science, 
Institute of Science Tokyo, 
2-12-1 Ookayama, Meguro-ku, Tokyo 152-8552 Japan}
\email{takahashi.j.cf61@m.isct.ac.jp}

\author[E. Zhanpeisov]{Erbol Zhanpeisov}
\address[Erbol Zhanpeisov]{Mathematical Institute, 
Graduate School of Science, 
Tohoku University, 
6-3 Aramaki Aza-Aoba, Aoba-ku, Sendai 980-8578, Japan}
\email{zhanpeisov.erbol.d6@tohoku.ac.jp}

\subjclass[2020]{Primary 35K58, Secondary 35B40, 35B44}
\keywords{Energy-critical heat equation, Infinite-time blow-up, Gluing method.}

\begin{abstract}
We consider the six dimensional energy-critical semilinear heat equation 
with self-similarly decaying initial data.
Our main result shows the existence of 
sign-changing solutions that exhibit infinite-time blow-up   
and nonnegative solutions that decay strictly more slowly than the self-similar rate. 
Moreover, the blow-up and decay rates are not uniquely determined 
by the decay rate of the initial data, 
but exhibit a certain flexibility depending on the construction.
The proof is based on gluing suitably rescaled bubbles 
to forward self-similar solutions. 
\end{abstract}

\maketitle

\section{Introduction}
We consider the following energy-critical semilinear heat equation: 
\begin{equation}\label{eq:SHE1}
\left\{ 
\begin{aligned}
	&\partial_t u  = \Delta u + |u|^{\frac{4}{n-2}}u 
	&&\mbox{ in }\mathbf{R}^n\times(0,\infty), \\
	&u(\cdot,0)=u_0 &&\mbox{ in }\mathbf{R}^n, 
\end{aligned}
\right. 
\end{equation}
where $n\geq3$. 
This is a critical case of 
$\partial_t u  = \Delta u + |u|^{p-1}u$ with $p>1$ 
and the equation with the nonlinearity $|u|^{p-1}u$  
is one of the most extensively studied 
scalar nonlinear parabolic equations concerning 
singularity formation and global-in-time behavior. 
It is known from  \cite{CGS89} that 
all positive solutions of the steady problem 
$\Delta u + u^{(n+2)/(n-2)}=0$ consist of 
the rescaled bubbles 
$U_{\lambda,\xi}(x):=\lambda^{-(n-2)/2} U((x-\xi)/\lambda)$ 
for $\lambda>0$ and $\xi\in \mathbf{R}^n$, where 
$U$ is the Aubin--Talenti bubble defined by 
\begin{equation}\label{eq:bubble}
	U(x):= \left( 1+ \frac{|x|^2}{n(n-2)} \right)^{-\frac{n-2}{2}}. 
\end{equation}
The bubbles play a crucial role in the construction of type II 
finite-time blow-up solutions of problem \eqref{eq:SHE1} 
for $n=3,4,5,6$, see \cite{dPMW19,dPMWZZpre,Ha20,Sc12}. 
Here, blow-up is said to be of type II if its blow-up rate 
is not consistent with that of $v_t=|v|^{4/(n-2)}v$.

The bubbles are also a key ingredient 
in the construction of global-in-time solutions. 
Related to this, Fila--King \cite{FK12} gave deep insight 
for the infinite-time blow-up and decaying 
behaviors of positive radial solutions with $u_0\in C(\mathbf{R}^n)$ satisfying 
\[
	\lim_{|x|\to \infty} |x|^{\gamma'} u_0(x)= \ell'
	\quad \mbox{ for some }0<\ell'<\infty, \; \gamma'>\frac{n-2}{2}, 
\]
where its decay rate is faster than 
the self-similar decay $|x|^{-(n-2)/2}$. 
By using bubbles, 
the behaviors predicted in \cite[Conjecture 1.1]{FK12} 
were first confirmed by 
del Pino--Musso--Wei \cite{DMW20} for $n=3$ with $\gamma'>1$, 
and then by \cite{LWZZ24,WZZ24,WZ25} for $n=4,5,6$. 
For the remaining cases $n=3$ with $1/2<\gamma'\leq 1$, and $n\geq7$, 
see \cite[Remark 1.2]{WZ25}. 
For related results on global-in-time dynamics, 
see \cite{AD25,GK03,GV97,Ka96,PY03,PY14,Qu08}. 

For $n=6$, Harada \cite{Ha25pre} considered 
initial data $u_0\in \dot{H}^1(\mathbf{R}^6)$ with the 
decay rate $|x|^{-2}(\log |x|)^{-c}$ ($1/2<c<1$)  
and the nontrivial modifications of such $u_0$. 
He constructed positive solutions 
that decay more slowly than the self-similar rate $t^{-1}$,
sign-changing solutions that exhibit infinite-time blow-up 
and sign-changing solutions 
whose absolute values oscillate between $0$ and $\infty$ as $t\to\infty$. 
We note that the solutions in \cite{Ha25pre} 
belong to the energy class $C([0,\infty); \dot{H}^1(\mathbf{R}^6))$. 
To the best of our knowledge, for $u_0\not\in\dot{H}^1(\mathbf{R}^6)$, 
such peculiar solutions remain unknown.

In this paper, we also consider the case $n=6$ and restrict our attention 
to the case $u_0\not \in \dot{H}^1(\mathbf{R}^6)$. 
More precisely, we impose the explicit self-similar decay 
\begin{equation}\label{eq:SScond}
	\lim_{|x|\to \infty} |x|^2 u_0(x)= \ell_A,  
\end{equation}
where $\ell_A\neq0$ can be varied near $0$ 
and is determined as follows. 
For $A\neq0$, let $\theta_A$ be a 
forward self-similar solution of the equation in \eqref{eq:SHE1} for $n=6$ 
of the form 
\begin{equation}\label{eq:SSSprofile}
	\theta_A(x,t) 
	= (t+1)^{-1} \Theta_A\left(\frac{|x|}{\sqrt{t+1}}\right) 
	\quad \mbox{ for }x \in \mathbf{R}^6, \; t>0, 
\end{equation}
with the profile $\Theta_A(r)$ ($r>0$) 
satisfying $\Theta_A(0)=A$ and $\Theta_A'(0)=0$. 
Then, we set 
\[
	\ell_A:=\lim_{r\to \infty}r^2 \Theta_A(r). 
\]
Here, the existence of $\theta_A$ and $\ell_A$ 
follows from \cite[Theorem 5]{HW82}. 
Note that the sign of $\ell_A$ coincides with $A$. 
Moreover, $\ell_A\to 0$ as $A\to0$, see \cite[Theorem 1.1 (iii)]{Na06}. 

As notation, for functions $f_1, f_2\geq 0$, 
we write $f_1 \lesssim f_2$ if $f_1(t) \le C f_2(t)$ for all sufficiently large $t$. 
We write $f_1 \ll f_2$ if $f_1(t)/f_2(t) \to 0$ as $t\to\infty$. 
For functions $g_1$ and $g_2$, 
we write $g_1 \sim g_2$ if $|g_1/g_2|\lesssim 1$ and $|g_1/g_2|\gtrsim 1$
for all sufficiently large $t$. 
We denote the characteristic function on a set $D$ by $\mathbf{1}_D$. 
Let $\eta\in C^\infty_0(\mathbf{R}^6)$ satisfy $0\le \eta \le 1$, 
$\eta(x) =1$ for $|x| \le 1$ and 
$\eta (x)=0$ for $|x| \ge 2$. 

Our main result is the following: 

\begin{theorem}\label{Theorem:A}
Let $n=6$ and $1/2<a_1<a<1$. 
Then there exists $0<A_0<2/5$ such that the following holds 
for each $-A_0<A<A_0$: there exists $u_0\in L^\infty(\mathbf{R}^n)$ 
satisfying \eqref{eq:SScond} such that problem\eqref{eq:SHE1} 
has a radially symmetric global-in-time solution $u$ of the form 
\[
	u(x,t) = 
	\lambda(t)^{-2} \left(
	1+\frac{|x|^2}{24\lambda(t)^2}
	\right)^{-2} 
	\eta\left( \frac{x}{\sqrt{t}}\right) 
	+ \theta_A(x,t) + v(x,t)
\]
for all $x\in \mathbf{R}^6$ and $t>0$. 
Here, $\lambda$ and $v$ satisfy that 
\[
\begin{aligned}
	&\lambda(t) 
	\sim 
	t^{\frac{5A}{4}} \ll \sqrt{t}
	&&\mbox{ for }t\gg1, \\
	&\begin{aligned}
	|v(x,t)|
	&\lesssim 
	t^{-1}(\log(e+t))^{2(7-a)} \mathbf{1}_{|x|\le \sqrt{t}} \\
	&\quad 
	+ (\log(e+t))^{-2a_1}|x|^{-2}\mathbf{1}_{|x|\ge \sqrt{t}} 
	\end{aligned}
	&&\mbox{ for }x\in\mathbf{R}^6, \; t\gg1, \\
	&|v(x,t)| \ll \left| \lambda(t)^{-2} \left(
	1+\frac{|x|^2}{24\lambda(t)^2}
	\right)^{-2} 
	\eta\left( \frac{x}{\sqrt{t}}\right) 
	+ \theta_A(x,t) \right| 
	&&\mbox{ for }x \in \mathbf{R}^6, \; t\gg1. 
\end{aligned}
\]
Moreover, $u$ is sign-changing in the infinite-time blow-up case $-A_0<A<0$ 
and is nonnegative in the slowly decaying case $0<A<A_0$. 
\end{theorem}

Theorem \ref{Theorem:A} gives the first construction 
of infinite-time blow-up solutions and slowly decaying solutions 
of problem \eqref{eq:SHE1} 
with self-similarly decaying initial data. 
Concerning the decaying case, 
we recall from Naito \cite[Theorem 1.1]{Na20} that 
there exists $\tilde \ell_*>0$ such that 
if the initial data $\tilde u_0\in C(\mathbf{R}^6)$ satisfies 
\begin{equation}\label{eq:Nacond}
	\lim_{|x|\to\infty} |x|^2 \tilde u_0(x)=\tilde \ell, 
	\quad 
	0<\tilde u_0(x)\leq \tilde \ell_*|x|^{-2} \mbox{ for }x\in \mathbf{R}^6 
\end{equation}
with some $0<\tilde \ell<\tilde \ell_*$, 
then the solution exists globally-in-time 
and has the self-similar decay $t^{-1}$. 
Our $u_0$ lies slightly outside \eqref{eq:Nacond} and 
the solution decays strictly slower than $t^{-1}$. 
Therefore, we provide the first example of solutions 
whose initial data exhibit self-similar decay, while the solutions
do not follow the self-similar decay rate. 
Moreover, the rate can be chosen to a certain extent. 
These imply that even a slight deviation from \eqref{eq:Nacond}
can lead to uncommon behaviors. 

The known solutions in \cite{DMW20,Ha25pre,LWZZ24,WZZ24,WZ25} 
for $u_0\in \dot{H}^1(\mathbf{R}^n)$ 
have been constructed by the parabolic inner-outer gluing method 
developed by Cort\'azar--del Pino--Musso \cite{CDM20}. 
Under $u_0\in \dot{H}^1(\mathbf{R}^n)$, they glue a suitably rescaled bubble 
to a solution of the linear heat equation. 
In contrast, we glue the bubble to 
a forward self-similar solution of the semilinear heat equation. 
This is an essential modification, 
which allows us to handle the self-similar decay $|x|^{-2}$ in \eqref{eq:SScond}. 
Due to the ansatz in Section \ref{sec:system} 
(see in particular Remark \ref{Remark:reason6D}), 
our argument works only in the case $n=6$. 
For the overview of the gluing construction 
in this paper, see Appendix \ref{sec:blue}. 

By the proof, we see that 
the solution $u$ is positive further inside the inner region $|x|\ll \lambda(t)$, 
behaves as $u(x,t)\sim \lambda(t)^{-2}U(x/\lambda(t))$ 
in the inner region $|x|\lesssim \lambda(t)$ and 
behaves as $u(x,t)\sim \theta_A(x,t)$ 
in the self-similar region $|x|\sim \sqrt{t}$. 
Moreover, when $A<0$, the solution 
is negative away from the self-similar region $|x|\gtrsim \sqrt{t}$. 
We note that 
our proof also works 
for constructing solutions of the form 
\[
	u(x,t) = 
	\lambda(t)^{-2} \left(
	1+\frac{|x-\xi(t)|^2}{24\lambda(t)^2}
	\right)^{-2} 
	\eta\left( \frac{x-\xi(t)}{\sqrt{t}}\right) 
	+ \theta_A(x,t) + v(x,t), 
\]
where $\xi(t)$ is determined by an orthogonality condition 
and satisfies growth estimates on $|\xi(t)|$ and $|\xi'(t)|$. 
However, we do not know whether the resultant solution  
is actually nonradial or not, 
since the authors could not exclude the case $\xi(t)\not \equiv0$. 

The rest of this paper is organized as follows. 
In Section \ref{sec:system},  
we derive the inner--outer gluing system. 
In Section \ref{sec:choice}, we choose an appropriate 
modulation parameter $\lambda$. 
In Section \ref{sec:inner}, we consider the inner problem. 
In Section \ref{sec:outer}, we solve the outer problem and 
complete the proof of Theorem \ref{Theorem:A}. 
For the reader's convenience, we summarize 
an overview of the gluing construction  
in Appendix \ref{sec:blue}.

\section{Inner--outer gluing system}\label{sec:system}
In the rest of this paper, we set $n=6$. 
For the reason why the argument does not work when 
$n\neq 6$, see Remark \ref{Remark:reason6D}. 
We construct the desired solutions by solving 
\begin{equation}\label{eq:6Dt0}
	\partial_t u - \Delta u = |u|u 
	\quad \mbox{ in }\mathbf{R}^6\times(t_0,\infty) 
	\mbox{ with }t_0\gg1. 
\end{equation}
Here and below, 
we write $t_0\gg1$ if $t_0$ is sufficiently large. 
Once we construct solutions of this equation 
with appropriate initial data $u(\cdot,t_0)$, 
Theorem \ref{Theorem:A} follows from shifting the time variable. 
Then, we seek solutions of \eqref{eq:6Dt0} of the form 
\begin{equation}\label{eq:Ansatzu} 
	u(x,t) 
	= 
	\underbrace{ 
	\lambda^{-2}(t) U\left( y \right) 
	\eta\left( \tilde{y} \right) 
	+ \theta_A(x,t) 
	}_{\text{leading terms}} 
	+ \underbrace{ 
	\underbrace{ 
	\lambda^{-2}(t) \phi\left( y, t \right) 
	\eta_R(y,t) 
	}_{\text{an inner profile}}
	+\underbrace{ \psi(x,t), 
	}_{\text{an outer profile}}
	}_{\text{remainder terms}}
\end{equation}
where $\phi$, $\psi$ and $\lambda$ are unknown functions, 
the modulation parameter $\lambda$ will be chosen later by an orthogonality condition  
and $y$ and $\tilde y$ are given by 
\[
	y = \frac{x}{\lambda(t)}, \quad  \tilde{y} = \frac{x}{\sqrt{t}}. 
\]
Moreover, $U(y) = ( 1 + |y|^2/24 )^{-2}$ is the bubble 
defined by \eqref{eq:bubble} with $n=6$. 
Recall that $\eta\in C^\infty_0(\mathbf{R}^6)$ 
satisfies $\eta(\tilde y) =1$ for $|\tilde y| \le 1$, $\eta (\tilde y)=0$ 
for $|\tilde y| \ge 2$ 
and $0\le \eta(\tilde y) \le 1$ for $\tilde y\in \mathbf{R}^6$. 
Set $\eta_R(y,t) := \eta(y/R(t))$, 
where $R(t)$ is defined by
\begin{equation}\label{eq:defR}
	R(t):=(\log(e+t))^2.
\end{equation}

Substituting $u$ into \eqref{eq:6Dt0} with straightforward computations gives 
\[
\begin{aligned}
	&\lambda^{-4} \left[ \lambda^2 \partial_t \phi(y,t) 
	- \Delta\phi(y,t) 
	- 2 U(y) \phi(y,t)\eta(\tilde{y}) \right] \eta_R(y,t)
	+ \partial_t\psi(x,t) -\Delta\psi(x,t)  \\
	&= 
	\lambda^{-4} \left[ 
	\lambda \dot{\lambda} (2 U(y) + y\cdot \nabla U(y))
	+ 2\lambda^2 U(y) \theta_A  
	+ 2\lambda^2 U(y) \psi 
	\right] \eta(\tilde{y}) \eta_R(y,t) \\
	&\quad +\mathcal{N}[\phi,\psi,\lambda](x,t) 
	+ {\mathcal{E}}[\phi,\psi,\lambda](x,t) 
	+ \tilde{\mathcal{E}}[\lambda](x,t) 
\end{aligned}
\]
with $\dot{\lambda}:= d\lambda/dt$. Here, omitting the variables 
when the function is evaluated at $(x,t)$ or $t$, we set 
\begin{align}
	&\label{eq:defofN}
	\begin{aligned}
	&\mathcal{N}[\phi,\psi,\lambda] 
	:=|\lambda^{-2} U\left( y \right) 
	\eta\left( \tilde{y} \right) 
	+ \theta_A
	+ \lambda^{-2} \phi( y, t ) 
	\eta_R(y,t) + \psi| \\
	&\times 
	(\lambda^{-2} U\left( y \right) 
	\eta\left( \tilde{y} \right) 
	+ \theta_A
	+ \lambda^{-2} \phi\left( y, t \right) 
	\eta_R(y,t) + \psi ) \\
	&
	-\lambda^{-4} 
	U(y)^2 \eta(\tilde{y})^2 
	- |\theta_A|\theta_A
	-2 \lambda^{-2} U(y) \eta(\tilde{y})
	(\theta_A+\psi+ \lambda^{-2}\phi(y,t)\eta_R(y,t)), \\
	\end{aligned} \\
	&\label{eq:cEdef}
	\begin{aligned}
	&{\mathcal{E}}[\phi,\psi, \lambda] 
	:= \lambda^{-4}R^{-2} \phi(y,t) \Delta\eta\left(\frac{y}{R}\right) 
	+ 2\lambda^{-4}R^{-1} \nabla\phi(y,t) 
	\cdot \nabla\eta\left(\frac{y}{R}\right)\\
	&+ \lambda^{-2} \phi(y,t) 
	\nabla\eta\left(\frac{y}{R}\right)\cdot
	\frac{y}{R}\frac{\partial_t(\lambda R)}{\lambda R}  
	+ \lambda^{-3}\dot{\lambda}\left(2\phi(y,t) 
	+ y\cdot \nabla \phi(y,t)\right)\eta_R(y,t) \\
	&
	+ 2\lambda^{-2} U(y) \psi\eta(\tilde{y})(1-\eta_R(y,t)) 
	\end{aligned} 
\end{align}
and 
\begin{equation}\label{eq:tilcEdef}
\begin{aligned}
	&\tilde {\mathcal{E}}[\lambda]:= 
	\lambda^{-3}\dot{\lambda} (2 U(y) + y\cdot \nabla U(y)) 
	\eta(\tilde{y}) (1-\eta_R (y,t))
	\\
	&+ \lambda^{-4}U(y)^2 (\eta(\tilde{y})^2 -\eta (\tilde{y})) 
	+ 2\lambda^{-2} U(y) \theta_A \eta(\tilde{y})(1-\eta_R(y,t)) \\
	&+2^{-1} t^{-1}\lambda^{-2} U(y) \tilde{y} 
	\cdot \nabla \eta (\tilde{y}) 
	+ 2\lambda^{-3}t^{-\frac{1}{2}}\nabla U(y) \cdot \nabla \eta (\tilde{y})
    + \lambda^{-2}t^{-1} U(y) \Delta\eta (\tilde{y}). 
\end{aligned}
\end{equation}

To separate the scales of $y$ and $\tilde y$, we may assume that 
\begin{equation}
\label{eq:Ansatzmu-}
	18\lambda(t) R(t) \le \sqrt{t} \quad \mbox{ for }t>t_0\gg1. 
\end{equation}
This gives $\eta(\tilde y) \eta_R(y,t) =\eta_R(y,t)$. 
Then, similarly to \cite{WZ25}, 
by noting the localization of $\eta_R(y,t)$, 
we obtain the following inner-outer gluing system 
consisting of the inner problem for $\phi$ 
and the outer problem for $\psi$: 
\begin{align}
	&\label{eq:Innerproblem-}
	\lambda^2 \partial_t \phi - \Delta_y \phi 
	- 2 U(y) \phi 
	= \mathcal{G}[\psi, \lambda], 
	&&
	y \in B_{4R(t)}, \; t>t_0, \\
	&\label{eq:Outerproblem-}
	\partial_t \psi - \Delta_x \psi = \mathcal{H}[\phi, \psi, \lambda], 
	&&x\in \mathbf{R}^6, \; t> t_0, 
\end{align}
where the data at $t=t_0$ will be chosen suitably and 
\begin{align}
&\label{eq:defofH-}
\begin{aligned}
	\mathcal{G}[\psi, \lambda]
	&:= \lambda \dot{\lambda}(2 U(y) + y\cdot \nabla U(y))
	+ 2\lambda^2 U(y) \theta_A(\lambda y, t) \\
	&\quad + 2\lambda^2 U(y) \psi(\lambda y,t), 
\end{aligned}\\
&\label{eq:defofG-}
	\mathcal{H}[\phi, \psi, \lambda] 
	:= \mathcal{N}[\phi, \psi, \lambda] 
	+ {\mathcal{E}}[\phi, \psi, \lambda] 
	+ \tilde {\mathcal{E}}[\lambda]. 
\end{align}

We assume that $\lambda$ is decomposed as 
\begin{equation}\label{eq:ansatz_mu_0-}
	\lambda(t) = \lambda_0(t) + \mu(t), 
	\quad |\mu(t)| \le \frac{\lambda_0(t)}{9}, 
	\quad |\dot{\mu}(t)| \le \frac{|\dot{\lambda}_0(t)|}{9}. 
\end{equation}
Then, we determine the leading term $\lambda_0$ as follows, 
where $\mu$ will be fixed in Section \ref{sec:choice}. 
We recall that the linearized operator $\Delta + 2 U$ has
a bounded radial kernel 
\begin{equation}\label{eq:kernels}
	Z (x) := 2 U(x) + x\cdot \nabla U(x), 
\end{equation}
where  $Z$ corresponds to  the scaling direction. 
Then, the orthogonality condition for the inhomogeneous term $\mathcal{G}$ 
in \eqref{eq:Innerproblem-} 
with respect to the scaling kernel $Z$ is 
\begin{equation}\label{eq:orthogonalitycondition}
	\int_{\R^6}
	\left[ \lambda_0(t) \dot{\lambda}_0(t) Z(y) 
	+ 2\lambda_0(t)^2 U(y) \theta_A(0,t) \right] Z(y)\eta_{4R(t)}(y)  dy 
	=0. 
\end{equation}
Since $\theta_A(0,t)=A(t+1)^{-1}$ by \eqref{eq:SSSprofile} and $\Theta_A(0)=A$, 
it suffices that $\lambda_0$ satisfies 
\begin{equation}\label{eq:orthogonalitycondition_mu_0-}
	\dot{\lambda}_0(t)  
	=-\frac{2\int_{\R^6} U(y) Z(y)\eta_{4R(t)}(y) dy}
	{\int_{{\R^6}} Z(y)^2\eta_{4R(t)}(y) dy} 
	A(t+1)^{-1} \lambda_0(t). 
\end{equation}
Fix any positive solution of \eqref{eq:orthogonalitycondition_mu_0-}. 
Note that $\lambda_0$ is determined at this moment. 
The orthogonality condition \eqref{eq:orthogonalitycondition} determines 
the choice of $\lambda$ in Section \ref{sec:choice}, 
which guarantees the solvability of the inner problem \eqref{eq:Innerproblem-} 
in Section \ref{sec:inner}. 

To derive the asymptotics of $\lambda_0$, 
we claim that 
\begin{equation}\label{eq:asymptfrac}
\begin{aligned}
	-\frac{2\int_{\R^6} U(y)Z(y)\eta_{4R(t)}(y) dy}{\int_{\R^6}Z(y)^2 \eta_{4R(t)}(y) dy} 
	&=
	\frac{2\int_{\mathbf{R}^6} U(y)^2  dy}{\int_{\mathbf{R}^6}Z(y)^2 dy} 
	+  O(R(t)^{-2})\\
	&=
	\frac{5}{4}+ O(R(t)^{-2})
	\quad \mbox{ as } t\to \infty.\\
\end{aligned}
\end{equation}
Indeed, by $U(y) = ( 1+ |y|^2/24)^{-2}$, 
$Z(y) = 2U(y) + y\cdot \nabla U(y)$ and 
integration by parts, we can compute that for $r>0$, 
\[
\begin{aligned}
	&- \frac{2 \int_{B_r} U(y) Z(y) dy}{\int_{B_r}Z(y)^2 dy} 
	= \frac{2 \int_{\mathbf{R}^6} U(y)^2dy 
	- 2 \int_{\mathbf{R}^6\setminus B_r} U(y)^2dy
	-r \int_{\partial B_r} U(y)^2 dS}
	{ \int_{\mathbf{R}^6}Z(y)^2 dy
	- \int_{\mathbf{R}^6\setminus B_r}Z(y)^2 dy } \\
	&= \frac{2\|U\|_{L^2(\mathbf{R}^6)}^2 - 2 ( (24^4/2) \pi^3 r^{-2} + O(r^{-4}) )
	- r (24^4 \pi^3 r^{-3} + O(r^{-5})) }
	{\|Z\|_{L^2(\mathbf{R}^6)}^2 - 2\cdot 24^4 \pi^3r^{-2} + O(r^{-4})}, 
\end{aligned}
\]
where $dS$ is the surface area element. 
Since $\|U\|_{L^2(\mathbf{R}^6)}^2 = 24^3\pi^3 /6$ 
and $\|Z\|_{L^2(\mathbf{R}^6)}^2 = 4\times 24^3\pi^3/15$, 
we obtain \eqref{eq:asymptfrac}. 

By \eqref{eq:orthogonalitycondition_mu_0-} 
and \eqref{eq:asymptfrac}
we see that $\lambda_0$ must satisfy  
\[
	\dot{\lambda}_0(t) 
	= \Big(\frac{5A}{4}t^{-1}  
    +O(t^{-1}R(t)^{-2})\Big)\lambda_0(t)
\]
as $t\to \infty$. 
By $R(t)=(\log(e+t))^2$ in \eqref{eq:defR} and by taking $|A|\ll1$, 
\begin{equation}\label{eq:leadingmu_000-}
	\lambda_0(t) 
	\sim t^{\frac{5A}{4}}
	\ll \sqrt{t},
	\quad \dot{\lambda}_0(t) 
	\sim \frac{5A}{4}t^{-1+\frac{5A}{4}}, 
\end{equation}
for $t\gg1$.
Here and below, we write $|A|\ll1$ if $|A|$ is sufficiently small. 
We see from \eqref{eq:ansatz_mu_0-} 
that $\lambda$ also satisfies that 
\begin{equation}\label{eq:leadingmu_0-}
	\lambda(t) \sim t^{\frac{5A}{4}},
	\quad \dot{\lambda}(t) 
	\sim \frac{5A}{4}t^{-1+\frac{5A}{4}}
    \quad \mbox{ for } t\gg1.
\end{equation}
In particular, the separation assumption \eqref{eq:Ansatzmu-} is satisfied. 

\begin{remark}
\label{Remark:reason6D}
For general $n\ge 3$, one can observe from direct computations that 
$\mathcal{G}[\psi, \lambda]$ contains the terms 
\[
	\lambda \dot{\lambda}\left( \frac{n-2}{2} U(y) + y \cdot \nabla U(y) \right)
	+ \frac{n+2}{n-2}\lambda^\frac{n-2}{2}U(y)^\frac{4}{n-2} 
	\theta_A(\lambda y, t). 
\]
Then, the condition for $\lambda_0$ from the orthogonality is reduced to 
\[
	\dot{\lambda}_0 \sim C_n \lambda_0^\frac{n-4}{2} t^{-\frac{n-2}{4}} 
	\quad \mbox{ for } t\gg1,
\]
where $C_n>0$ is a suitable constant.
If $n\neq 6$, whatever $C_n$ may be, one has
\[
	\lambda_0(t) \sim \sqrt{t} 
	\quad \mbox{ for } t\gg1. 
\]
This implies that the matching region collapses 
compared with \eqref{eq:leadingmu_0-}, 
and so our construction works only for $n=6$. 
\end{remark}

Up to this point, $\lambda_0$ is already fixed, 
but $\mu$ remains to be determined. 
In Section \ref{sec:choice}, 
for given $\psi$ in an appropriate function space, 
we construct a suitable $\mu=\mu_\psi$ and 
a modulation parameter $\lambda=\lambda_\psi$. 
Then, by using the given $\psi$ and the constructed $\lambda_\psi$, 
we construct a solution $\phi=\phi_\psi$ of 
the inner problem \eqref{eq:Innerproblem-} 
in Section \ref{sec:inner}. 
Consequently, the outer problem \eqref{eq:Outerproblem-} 
becomes a single equation 
$\partial_t \psi - \Delta_x \psi = \mathcal{H}[\phi_\psi, \psi, \lambda_\psi]$ 
for $\psi$. 
Finally, we solve this equation 
in Section \ref{sec:outer}. 

To carry out the above strategy, we set up function spaces as follows. 
Let $t_0\gg1$ and $t_0<\tau< \infty$. 
For $\lambda=\lambda_0+\mu$, we set 
\begin{equation}\label{eq:assmpfor_mu1_xi-}
	\Lambda^\tau:= \left\{ 
	\lambda(t); \; \lambda=\lambda_0+\mu>0, 
	\; \mu  \in C^1([t_0,\tau]), \; 
	|\mu| \le \frac{\lambda_0}{9}, \; 
	|\dot \mu| \le \frac{|\dot{\lambda}_0|}{9}
	\right\}, 
\end{equation}
where $\lambda_0\in C^1([t_0,\infty))$ 
is given by \eqref{eq:orthogonalitycondition_mu_0-} 
and each $\lambda\in \Lambda^\tau$ satisfies \eqref{eq:leadingmu_0-}. 
For each $\lambda\in \Lambda^\tau$, we define 
\begin{equation}
\label{eq:defoftau-}
	\sigma(t) = \sigma_\lambda(t) := \int_{t_0}^t \frac{ds}{\lambda(s)^2}
	+ \frac{t_0}{\lambda_0(t_0)^2} \quad \mbox{ for } t_0\leq t\leq \tau. 
\end{equation}
We note from \eqref{eq:leadingmu_0-} that  
\begin{equation}\label{eq:behabioroftau-}
	\sigma(t) \sim t^{1- \frac{5A}{2}}, 
\end{equation}
where $1- 5A/2>0$ by $|A|\ll1$. 
From $R(t)=(\log(e+t))^2$ in \eqref{eq:defR}, 
it follows that
\begin{equation}
\label{eq:defoftauA-}
R(t) \le 2R(\sigma(t))\quad \mbox{ for } t_0\leq t\leq \tau. 
\end{equation}
Note that the choice of $A$ and $t_0$ can be determined independently of $\mu$.
Fix $1/2<a_1<1$. 
Define 
\begin{equation}\label{eq:spaceforouter-}
\begin{aligned}
	&B_{\rm out}^\tau := \left\{ \psi\in X_{\rm out}^\tau ; \; 
	\begin{aligned}
	&|\psi| \leq |A| w_{\rm out}, \; 
	|\nabla\psi| \leq |A| t^{-1}\lambda_0(t)^{-1} R(t)^{-1-a_1}\\ 
	&\mbox{ for }x \in \mathbf{R}^6, \; t_0\leq t\leq \tau
	\end{aligned}
	\right\}, \\
	&X_{\rm out}^\tau := \left\{ \psi(x,t);\; 
	\begin{aligned}
	&\mbox{$\psi(\cdot,t)$ is radially symmetric for each $t_0\leq t\leq \tau$,} \\
	&\psi \in C^{1,0}(\mathbf{R}^6\times[t_0,\tau]), 
	\; \|\psi\|_{\rm out}<\infty
	\end{aligned}
	\right\}, \\
	&\|\psi\|_{\rm out} 
	:= \sup_{x \in \mathbf{R}^6, \; t_0\leq t\leq \tau } |\psi(x,t)| 
	+ \sup_{x \in \mathbf{R}^6, \; t_0\leq t\leq \tau } |\nabla \psi(x,t)|, \\
	&w_{\rm out}(x,t) := t^{-1} R(t)^{-a_1}
	\left[\mathbf{1}_{|x| \le \sqrt{t}} 
	+ t|x|^{-2} \mathbf{1}_{|x| \ge \sqrt{t}}\right]. 
\end{aligned}
\end{equation}
We remark that $B_{\rm out}^\tau$ depends on $\lambda_0$ 
but not on  $\lambda$ 
and that $B_{\rm out}^\tau$ is a closed and convex subset of 
a Banach space $X_{\rm out}^\tau$.

\section{Choice of parameters}\label{sec:choice}
Throughout this section, we fix $\psi\in B_{\rm out}^\tau$. 
We determine an appropriate modulation parameter 
$\lambda$  to satisfy 
the following orthogonality 
condition \eqref{eq:orthogonalitycondition_for_ourproblem}. 
This will be done by solving an  ODE, the so called
modulation equation, based on 
the contraction mapping theorem. 
Throughout this section, we continue to fix 
$1/2<a_1<1$ 
in the definition of $B_{\rm out}^\tau$ in \eqref{eq:spaceforouter-}.
The main result in this section is as follows.

\begin{proposition}\label{Proposition:SolvabilityforOurinner}
Assume that $\tau > t_0 \gg1$, $1/2<a_1<1$ and $A\neq 0$ with $|A|\ll1$. 
For each $\psi\in B_{\rm out}^\tau$, 
there exists a unique $\lambda_\psi \in \Lambda^\tau$  
such that 
\begin{equation}\label{eq:orthogonalitycondition_for_ourproblem}
	\int_{\R^6}  \mathcal{G}[\psi, \lambda_\psi](y,t)Z(y)\eta_{4R(t)}(y) dy =0
	\quad \mbox{ for }t_0<t<\tau
\end{equation}
and  $\mu_\psi:= \lambda_\psi-\lambda_0$ satisfies
\begin{equation}\label{eq:assmpfor_mu1_xi2}
	|\mu_\psi|\leq |A|\lambda_0, 
	\quad |\dot{\mu}_\psi| \lesssim  |\dot{\lambda}_0| 
	\quad \mbox{ for }t_0<t<\tau, 
\end{equation}
where $\mathcal{G}$ is defined by \eqref{eq:defofH-}, 
$Z$ is defined by \eqref{eq:kernels} and 
$\lambda_0$ is given by \eqref{eq:orthogonalitycondition_mu_0-}. 
\end{proposition}

To prove Proposition \ref{Proposition:SolvabilityforOurinner}, 
we observe the left-hand side 
of \eqref{eq:orthogonalitycondition_for_ourproblem}. 
From \eqref{eq:defofH-} and \eqref{eq:kernels}, it follows that 
\[
\begin{aligned}
	\int_{\R^6} \mathcal{G} Z\eta_{4R(t)}(y) dy 
	&= \lambda \dot{\lambda} 
	\int_{\R^6} Z(y)^2 \eta_{4R(t)}(y) dy \\
	&+ 2\lambda^2 
	\int_{\R^6}
	( \theta_A(\lambda y, t)  
	+\psi(\lambda y,t) ) U(y) Z(y)\eta_{4R(t)}(y) dy. 
\end{aligned}
\]
By \eqref{eq:orthogonalitycondition_mu_0-}, 
the desired $\lambda$ will be constructed 
once we solve the following ODE: 
\[
\begin{aligned}
	&\dot{\mu}(t) 
	=  \beta(t) \mu(t)
	+ \tilde {\mathcal{F}}[\mu](t), \quad 
	\beta(t):=-\frac{2\int_{\R^6} U Z\eta_{4R} dy}
	{\int_{\R^6} Z^2\eta_{4R} dy} \theta_A(0,t)  \\
	&\tilde {\mathcal{F}}_\psi[\mu](t):= 
	\frac{- 2\int_{\R^6}
	( \theta_A(\lambda y, t) - \theta_A(0, t) 
	+\psi(\lambda y,t) ) U Z\eta_{4R} dy}
	{\int_{\R^6} Z^2 \eta_{4R}dy} \lambda(t), 
\end{aligned}
\]
with $\lambda=\lambda_0+\mu$ in the right-hand side of $\tilde{{\mathcal{F}}}_\psi$. 

We set up a fixed point problem for $\mu$. 
Define an operator $\mathcal{S}_\psi$ by 
\begin{equation}\label{eq:Spsidef}
	\mathcal{S}_{\psi}[\mu](t) :=
    \int_{t_0}^t \tilde {\mathcal{F}}_{\psi}[\mu](s) 
	e^{ \int_s^t \beta(\tau) d\tau} ds
\end{equation}
Here and below, we suppress the $\psi$-dependence of 
$\tilde {\mathcal{F}}_\psi$ and $\mathcal{S}_{\psi}$ for simplicity of notation. 
If we find a fixed point of $\mu=\mathcal{S}[\mu]$, 
then we obtain a solution $\mu$ 
of the following modulation equation:  
\begin{equation}\label{eq:modeq}
	\dot{\mu}(t) =  \beta(t) \mu(t)+ \tilde {\mathcal{F}}[\mu](t) \quad 
	\mbox{ for }t_0<t<\tau,
\end{equation}
where $t_0\gg1$. 
Moreover, $\lambda$ is also determined by $\lambda=\lambda_0+\mu$. 
In this way, the desired modulation parameter $\lambda$ 
will be constructed. 

To construct the desired fixed point of $\mathcal{S}$ 
by the contraction mapping theorem, 
we introduce a space
$B_{\rm sc}^\tau$ 
for  the modulation parameter $\mu$. 
Define 
\begin{equation}\label{eq:scdef}
\begin{aligned}
	&B_{\rm sc}^\tau := \{f \in C([t_0, \tau];\mathbf{R}); \|f\|_{\rm sc} \le |A|\}, \\
	&\|f\|_{\rm sc} 
	:= \sup_{t_0 \leq t\leq \tau} \lambda_0(t)^{-1} |f(t)|.
\end{aligned} 
\end{equation}

We first prepare  an estimate of a forward self-similar solution $\theta_A$. 

\begin{lemma}\label{thm:Psi_two_sided}
Let $n=6$ and $A\neq0$ with $|A|\ll1$.
Then $\theta_A$ satisfies
\begin{align}
	&\label{eq:SSS}
	\theta_A(x,t) \sim A \left[ t^{-1} \mathbf{1}_{|x| \le \sqrt{t}}
+ |x|^{-2}\mathbf{1}_{|x| > \sqrt{t}}\right],\\
    &\label{eq:SSS2}
    |\nabla \theta_A(x,t)| \lesssim |A|t^{-\frac{3}{2}},
\end{align}
for all $x \in \mathbf{R}^6$ and $t>1$.
\end{lemma}

\begin{proof}
It suffices to consider the case $A>0$,
since $\theta_A = -\theta_{|A|}$ when $A<0$. 
We recall that the profile $\Theta_A$ satisfies 
\begin{equation}
\label{eq:Theta_IVP}
\left\{
\begin{aligned}
	&\Theta_A'' + \frac{5}{r} \Theta_A' + \frac{r}{2}\Theta_A' 
	+ \Theta_A + |\Theta_A|\Theta_A=0,\quad  r>0, \\
	&\Theta_A(0)=A,\;  \Theta_A'(0)=0.  
\end{aligned}
\right.
\end{equation}
By introducing a linear operator
\[
\mathcal{L}f:=f''+\left(\frac{5}{r}+\frac{r}{2}\right)f' + f, 
\]
we have $\mathcal{L}\Theta_{A}+\Theta_{A}^2=0$.
Let $\tilde{\Theta}$ be the solution of
\[
	\mathcal{L}\tilde{\Theta} =0 \quad  \mbox{ for }r>0, 
	\quad 
	\tilde{\Theta}(0)=1,
	\quad 
	\tilde{\Theta}'(0)=0.  
\]
We see from \cite{SW19} that 
$\tilde{\Theta}$ exists and  can be represented as
$\tilde{\Theta}(r) =  e^{-r^2/4} 
M(2, 3; r^2/4)$,
where $M(\cdot, \cdot; \cdot)$ is Kummer's confluent hypergeometric function.
By \cite[Lemma A.1(ii)]{SW19}, 
we have $0< \tilde{\Theta}(r) \sim  (1+r)^{-2}$
for all $r\ge0$. 

First, let us prove \eqref{eq:SSS}. 
We introduce a ratio function $h(r)$ by
\[
	\Theta_A(r) = h(r)\tilde{\Theta}(r).
\]
Since $\Theta_A(0)=A$, we have $h(0)=A$. 
By $\mathcal{L}\Theta_{A}+\Theta_{A}^2=0$ and $\mathcal{L}\tilde{\Theta}=0$, 
\[
    \left( r^{5} e^{\frac{r^2}{4}} \tilde{\Theta}(r)^2 h'(r) \right)' = - r^{5} e^{\frac{r^2}{4}} \tilde{\Theta}(r) \Theta_A(r)^{2}.
\]
Integrating twice and using $h'(0)=0$ yield 
\begin{equation}\label{eq:hformula}
    h(r) = A 
    - \int_0^r \frac{e^{-\frac{\rho^2}{4}}}{\rho^{5}\tilde{\Theta}(\rho)^2}
    \left[ \int_0^\rho s^{5}e^{\frac{s^2}{4}} 
    \tilde{\Theta}(s) \Theta_A(s)^{2}  ds \right] d\rho.
\end{equation}
Since the integrand behaves like $\rho^2$ near $\rho=0$
and is positive, 
the integral on the right-hand side is well-defined for $r>0$ and 
$r\mapsto h(r)$ is nonincreasing. 
Therefore, $h(r) \le A$ for $r\ge0$.
This implies that
\begin{equation}\label{eq:SSSupper}
	\Theta_A(r) = h(r)\tilde{\Theta}(r)\le A\tilde{\Theta}(r)
	\lesssim A (1+ r)^{-2} \quad \mbox{ for } r\ge0.
\end{equation}
Since $0<\Theta_A\le A\tilde{\Theta}$, 
we have $\Theta_A^2\le A^2\tilde{\Theta}^2$. 
Then by \eqref{eq:hformula}, 
\[
	h(r)\ge A-A^p\,D(r),\quad 
	D(r):=\int_0^r \frac{e^{-\frac{\rho^2}{4}}}{\rho^{5}\tilde{\Theta}(\rho)^2}
    \left[ \int_0^\rho s^{5}e^{\frac{s^2}{4}} \tilde{\Theta}(s)^{3}   ds \right] d\rho.
\]

We claim that $\sup_{r\ge 0}D(r)<\infty$. 
From $\tilde{\Theta}(s)^{3}\lesssim (1+s)^{-6}$, it follows that 
\[
\int_0^\rho s^{5}e^{\frac{s^2}{4}} \tilde{\Theta}(s)^{3}   ds \lesssim
\int_1^\rho s^{-1}e^{\frac{s^2}{4}}ds \lesssim\rho^{-2}e^{\frac{\rho^2}{4}} 
\quad \mbox{ for } \rho\gg1, 
\]
where we used the integration by parts.
On the other hand,
\[
\frac{e^{-\frac{\rho^2}{4}}}{\rho^5\tilde{\Theta}(\rho)^2}\ \lesssim\ e^{-\frac{\rho^2}{4}}\rho^{-5}(1+\rho)^{4} \lesssim e^{-\frac{\rho^2}{4}}\rho^{-1}.
\]
Then, these estimates show the claim. Consequently, 
\[
	h(r)\ge A-CA=A\big(1-CA\big)\quad \mbox{ for } r\ge 0.
\]
Multiplying by $\tilde{\Theta}$ and choosing $A>0$ 
so small that $1-CA\ge 1/2$, we obtain
\begin{equation}
\label{eq:SSSlower}
	\Theta_A(r) = h(r)\tilde{\Theta}(r) 
	\gtrsim A\tilde{\Theta}(r)\gtrsim A(1+r)^{-2} 
	\quad \mbox{ for } r\ge0.
\end{equation}
By combining \eqref{eq:SSSupper} and \eqref{eq:SSSlower}, 
for $x\in\mathbf{R}^6$ and $t>1$, 
\[
\theta_A(x,t) = t^{-1}\Theta_A\left( \frac{r}{\sqrt{t}}\right)
\sim At^{-1}\left(1+ \frac{r}{\sqrt{t}}\right)^{-2}
\sim A \left[ t^{-1} \mathbf{1}_{|x| \le \sqrt{t}}
+ |x|^{-2}\mathbf{1}_{|x| > \sqrt{t}}\right]. 
\]

Next, we prove \eqref{eq:SSS2}.
By \eqref{eq:Theta_IVP} and \eqref{eq:SSS}, for sufficiently small $A>0$,
\[
\begin{aligned}
|\Theta_A'(r)| 
&\le r^{-5}e^{-\frac{r^2}{4}} \int_0^r s^5 e^\frac{s^2}{4} (\Theta_A(s) + \Theta_A(s)^2) ds\\
&\lesssim (A+A^2) r^{-5}e^{-\frac{r^2}{4}} \int_0^r s^5 e^\frac{s^2}{4}  ds 
\lesssim A, 
\end{aligned}
\]
and so 
\[
|\nabla \theta_A(x,t)| \le t^{-\frac{3}{2}} \left|\Theta_A'\left(\frac{|x|}{\sqrt{t}} \right) \right|
\lesssim A  t^{-\frac{3}{2}}. 
\]
Then the lemma follows. 
\end{proof}

Next, we check that $\mathcal{S}$ is a contraction mapping from $B_{\rm sc}^\tau$ 
into itself. 

\begin{lemma}\label{lem:selfcS1}
Assume that $\tau > t_0 \gg1$.
If $|A|\ll1$, 
then $\mathcal{S}$ is a contraction mapping from $B_{\rm sc}^\tau$ 
into itself. 
\end{lemma}

\begin{proof}
Let $\mu\in B_{\rm sc}^\tau$. 
We note that $\|\mu\|_{\rm sc}\leq |A|$. 
By using \eqref{eq:scdef} 
and by $|A|\ll 1$, $|\mu| \leq  |A|\lambda_0 \sim |A|t^{5A/4}$,
where $\lambda_0$ behaves as in \eqref{eq:leadingmu_000-}. 
This together with the mean value theorem,
\eqref{eq:SSS2}, \eqref{eq:spaceforouter-} and \eqref{eq:kernels} yields 
for $|A|\ll1$ and $1\ll t_0\leq t\leq \tau$, 
\begin{equation}\label{eq:EstforvecS}
\begin{aligned}
	&\begin{aligned}
	|\tilde {\mathcal{F}}|&\lesssim 
	\left(\lambda R \| \nabla \theta_A(\cdot, t)\|_{L^\infty(\mathbf{R}^6)} 
	+ \|\psi( \cdot,t)\|_{L^\infty(\mathbf{R}^6)}\right) 
	\frac{\int_{\R^6}U |Z|\eta_{4R} dy }{ \int_{\R^6} Z^2\eta_{4R} dy} \lambda \\
	&\lesssim 
	\left( \lambda_0 R |A| t^{-\frac{3}{2}} 
	+  |A| t^{-1} R^{-a_1}\right) 
    \lambda_0
    \lesssim |A|t^{-1+ \frac{5A}{4}}R^{-a_1}.
	\end{aligned} 
\end{aligned}
\end{equation}
By \eqref{eq:asymptfrac} with \eqref{eq:SSS}, 
\begin{equation}\label{eq:Estforbeta}
	\beta(t)= 
	-\frac{2\int_{\R^6} UZ\eta_{4R} dy}{\int_{\R^6}Z^2\eta_{4R} dy} 
	\theta_A(0,t) 
	= t^{-1}\left(\frac{5A}{4}+ O(R^{-2})\right). 
\end{equation}
Then by using \eqref{eq:EstforvecS}, \eqref{eq:Estforbeta} 
and \eqref{eq:defR}, we can estimate ${\mathcal{S}}[\mu](t)$ 
in \eqref{eq:Spsidef} as
\begin{equation}\label{eq:EstforF}
	|{\mathcal{S}}[\mu](t)| \le C|A|t^{\frac{5A}{4}} \int_{t_0}^\infty s^{-1}R(s)^{-a_1}ds \le |A|\lambda_0(t)
    \quad \mbox{ for } 1\ll t_0 <t<\tau, 
\end{equation}
and so 
$\mathcal{S}[\mu] \in B_{\rm sc}^\tau$. 
Taking $|\nabla \psi| \le |A|t^{-1}\lambda_0(t)^{-1} R(t)^{-1-a_1}$ 
into account, 
we can prove that $\mathcal{S}$ is a contraction mapping in the same way 
as for \eqref{eq:EstforvecS}. 
\end{proof}

Let us prove 
Proposition \ref{Proposition:SolvabilityforOurinner}. 

\begin{proof}[Proof of Proposition \ref{Proposition:SolvabilityforOurinner}]
By Lemma \ref{lem:selfcS1}, 
we can apply the contraction mapping theorem to 
$\mathcal{S}:B_{\rm sc}^\tau \to B_{\rm sc}^\tau$. 
Then, there exists a unique solution 
$\mu\in B_{\rm sc}^\tau$ of \eqref{eq:modeq}.
Set $\lambda:=\lambda_0+\mu$. The derivation of \eqref{eq:modeq} 
shows \eqref{eq:orthogonalitycondition_for_ourproblem}. 
The estimates in \eqref{eq:assmpfor_mu1_xi2} immediately 
follow from the definition of $B_{\rm sc}^\tau$ and \eqref{eq:modeq}. 
The proof is complete.
\end{proof}

Finally in this section, we show  
a kind of stability of $\lambda_\psi= \lambda_0 + \mu_{\psi}$. 

\begin{lemma}
\label{Lemma:stabilityofparameters}
Assume that $\tau > t_0 \gg1$. Let $\psi_1, \psi_2 \in B_{\rm out}^\tau$. 
Then $\|{\lambda}_{\psi_1}- {\lambda}_{\psi_2}\|_{\rm sc}
\to 0$
as $\psi_1\to \psi_2$ in the sense of $X_{\rm out}^\tau$.
\end{lemma}

\begin{proof}
Suppose that  $\psi_1, \psi_2 \in B_{\rm out}^\tau$ satisfy $\psi_1\to \psi_2$ in the sense of $X_{\rm out}^\tau$.
Since $1\ll t_0<\tau <\infty$,
in the same way as for \eqref{eq:EstforvecS}, \eqref{eq:Estforbeta} and \eqref{eq:EstforF}, we obtain for any $\mu \in B_{\rm sc}^\tau$, 
$\|\mathcal{S}_{\psi_1}[{\mu}]- \mathcal{S}_{\psi_2}[{\mu}]\|_{\rm sc}\to 0$ as $\|\psi_1 - \psi_2\|_{X_{\rm out}^\tau} \to 0$.

We write $\lambda_{\psi_i} = \lambda_0 + \mu_{\psi_i}$  for $i=1,2$.
Since $\mathcal{S}_{\psi_1}$ is a contraction mapping, 
there exists $0<\delta<1$ independent of $\psi_1$ such that
\[
\begin{aligned}
\|{\lambda}_{\psi_1}- {\lambda}_{\psi_2}\|_{\rm sc}
&=\|{\mu}_{\psi_1}- {\mu}_{\psi_2}\|_{\rm sc} 
= 
\|\mathcal{S}_{\psi_1}[{\mu}_{\psi_1}]- \mathcal{S}_{\psi_2}[{\mu}_{\psi_2}]\|_{\rm sc}\\
&\le \|\mathcal{S}_{\psi_1}[{\mu}_{\psi_1}]- \mathcal{S}_{\psi_1}[{\mu}_{\psi_2}]\|_{\rm sc} +
\|\mathcal{S}_{\psi_1}[{\mu}_{\psi_2}]- \mathcal{S}_{\psi_2}[{\mu}_{\psi_2}]\|_{\rm sc}\\
&\le \delta \|{\mu}_{\psi_1}- {\mu}_{\psi_2}\|_{\rm sc} 
+ \|\mathcal{S}_{\psi_1}[{\mu}_{\psi_2}]- \mathcal{S}_{\psi_2}[{\mu}_{\psi_2}]\|_{\rm sc}.
\end{aligned}
\]
This leads to the desired conclusion. 
\end{proof}

Up to this point, for given $\psi\in B_{\rm out}^\tau$, 
the modulation parameter 
$\lambda_\psi= \lambda_0 + \mu_{\psi}$ has been determined, 
but $\phi$ remains to be determined. 
In the next section, we construct $\phi$ 
for given $\psi\in B_{\rm out}^\tau$. 

\section{Inner problem}\label{sec:inner}
Let us construct $\phi$ for a given $\psi$ 
by using the orthogonality condition 
\eqref{eq:orthogonalitycondition_for_ourproblem}. 
We set up a framework for the inner problem.
Assume that $1\ll t_0 <\tau <\infty$. 
Fix $a_1<a<1$, where $1/2<a_1<1$ is already fixed 
just before \eqref{eq:spaceforouter-}. 
By using $\sigma$ in \eqref{eq:defoftau-}, we define
\begin{equation}\label{eq:spaceforinner-}
\begin{aligned}
	&B_{\rm in}^\tau := \left\{ \phi\in X_{\rm in}^\tau; 
	\; \|\phi\|_{\rm in} \le 1 \right\}, \\
	&X_{\rm in}^\tau := \left\{ \phi(y,s); \; 
	\begin{aligned} 
	&\mbox{$\phi(\cdot,s)$ is radially symmetric for each $\sigma(t_0)\leq s\leq \sigma(\tau)$,} \\
	&\phi \in C^{1,0}\left( \bigcup_{\sigma(t_0)\leq s\leq \sigma(\tau)} (B_{2R(s)} 
	\times \{s\}) \right),  
	\; \|\phi\|_{X_{\rm in}^\tau} <\infty
	\end{aligned}
	\right\}, \\
	&\|\phi\|_{\rm in}
	= \|\phi\|_{X_{\rm in}^\tau} 
	:= \sup_{y \in B_{2R(s)}, \; \sigma(t_0)\leq s\leq \sigma(\tau)} 
	\frac{ \langle y\rangle|\nabla \phi(y,s)| + |\phi(y,s)|}{w_{\rm in}(y,s)}, \\
	&w_{\rm in}(y,s):= s^{-1} R(s)^{7-a} \langle y\rangle^{-7},  
	\quad \langle y \rangle := \sqrt{1+|y|^2}. 
\end{aligned}
\end{equation}
We note that $\sigma$ is determined by $\lambda$ 
and that $B_{\rm in}^\tau$ is a closed and convex subset 
of a Banach space $X_{\rm in}^\tau$.

We recall that the linearized operator $\Delta + 2 U$ has 
a unique positive eigenvalue $\gamma_0$ 
and a corresponding radial and exponentially decaying eigenfunction 
$Z_0 \in L^\infty (\mathbf{R}^6)$ which satisfies 
$\Delta Z_0 + 2 U Z_0 = \gamma_0 Z_0$ in $\mathbf{R}^6$. 

\begin{proposition}\label{pro:solvinner}
Assume that $\tau > t_0 \gg1$, $1/2<a_1<a<1$ and $A\neq 0$ with $|A|\ll1$. 
Let $\psi\in B_{\rm out}^\tau$ and let 
$\lambda_\psi \in \Lambda^\tau$  
be as in Proposition \ref{Proposition:SolvabilityforOurinner}. 
Then there exists a unique pair 
$(\phi_\psi,C_\psi)\in B_{\rm in}^\infty \times \mathbf{R}$ 
such that $\phi_\psi$ satisfies the inner problem 
\begin{equation}\label{eq:innpro}
\left\{
\begin{aligned}
	&\begin{aligned}
	&\lambda_\psi^2 \partial_t \phi_\psi 
	- \Delta_y \phi_\psi 
	- 2 U(y) \phi_\psi  \\
	&\quad =  \mathcal{G}[\psi, \lambda_\psi] 
    \eta_{4R(t)}
	\mathbf{1}_{t\leq \tau}, 
	\end{aligned}
	&&y \in B_{{16R(\sigma_\psi(t))}}, \; t>t_0, \\
	&\phi_\psi(y,t_0) =  C_\psi Z_0(y), 
	&&y \in B_{16R(\sigma_\psi(t_0))}, 
\end{aligned}
\right.
\end{equation}
where $\mathcal{G}$ is defined by \eqref{eq:defofH-} 
and $B_{\rm in}^\infty$ is defined 
by replacing $t_0\leq t\leq \tau$ with $t_0\leq t<\infty$ 
in \eqref{eq:spaceforinner-}. 
In addition, $\sigma_\psi=\sigma_{\lambda_\psi}$ 
is given by \eqref{eq:defoftau-} with $\lambda=\lambda_\psi$. 
\end{proposition}

Thanks to \eqref{eq:defoftauA-}, the solution $\phi_\psi$ given by the above proposition is also a solution of the inner problem \eqref{eq:Innerproblem-} for $t_0<t\leq \tau$.
Although $\psi\in B_{\rm out}^\tau$ is defined up to $t\leq \tau$, 
the linear theory used in the proof of Proposition \ref{pro:solvinner} 
requires that the inhomogeneous term is defined for all $t> t_0$. 
For this reason, we used $\mathbf{1}_{t\leq \tau}$ in \eqref{eq:innpro}. 
To prove Proposition \ref{pro:solvinner}, 
we prepare a norm 
\[
	\|f\|_* := 
	\sup_{y \in B_{8R(s)}, \; s > s_0} 
	s \langle y\rangle^{2+a} |f(y,s)|. 
\]
We define $B_{\rm in}^\infty$ 
by replacing $t_0\leq t\leq \tau$ with $t_0\leq t<\infty$ 
in \eqref{eq:spaceforinner-}. 
Then, we recall the linear theory developed in 
\cite{CDM20,WZZ24,WZ25}. 

\begin{lemma}\label{lem:solvinn}
Assume that $t_0\gg1$. 
Let $g(y,t)$ satisfy $\|g\|_*<\infty$ and 
\[
	\int_{\R^6} g(y,t) Z(y)  dy =0
	\quad \mbox{ for }t>t_0, 
\]
where $Z$ is defined by \eqref{eq:kernels}. 
Then there exists a unique pair 
$(\tilde \phi, \tilde C)\in B_{\rm in}^\infty\times \mathbf{R}$ such that 
\begin{equation}\label{eq:inhom}
\left\{
\begin{aligned}
	&\lambda^2 \partial_t \tilde \phi 
	- \Delta_y \tilde \phi 
	- 2 U(y) \tilde \phi 
	= g(y,t), 
	&&y \in B_{16R(\sigma(t))}, \; t>t_0, \\
	&\tilde \phi(y,t_0) = \tilde C Z_0(y), 
	&&y \in B_{16R(\sigma(t_0))}. 
\end{aligned}
\right.
\end{equation} 
Moreover, $\tilde \phi(y,t) = \tilde \phi(y,\sigma^{-1}(s)) \in B_{\rm in}^\infty$ and $\tilde C\in\mathbf{R}$ satisfy 
\[
	\|\tilde \phi\|_{\rm in} \lesssim \|g\|_*, 
	\quad 
	|\tilde C|\lesssim \sigma(t_0)^{-1} R(t_0)^{2-a}\|g\|_*.
\]
If $g(\cdot,t)$ is radially symmetric for each $t>t_0$, 
then $\tilde \phi(\cdot,t)$ is also radially symmetric for each $t>t_0$. 
\end{lemma}

\begin{proof}
By using $\sigma$ in \eqref{eq:defoftau-}, 
we introduce 
\[
	\Phi(y,s) := \tilde \phi(y,\sigma^{-1}(s))= \tilde \phi (y,t), 
	\quad 
	s = \sigma (t). 
\]
Then, the left-hand sides in \eqref{eq:inhom} are transformed into 
\begin{equation}\label{eq:Innerproblem_tau-}
\left\{ 
\begin{aligned}
	&\partial_s  \Phi - \Delta_y \Phi
	- 2U(y) \Phi  
	= g(y,t) 
	&&y\in B_{16R(s)},\; s>s_0:= \sigma(t_0), \\
	&\Phi(y,s_0) = \tilde C Z_0(y), 
	&&y \in B_{16R(s_0)}, 
\end{aligned}
\right.
\end{equation}
where $\Phi$ in the left-hand side of the equation is evaluated at $(y,s)$. 
We note that $\lambda^2$ in \eqref{eq:inhom} 
vanishes in \eqref{eq:Innerproblem_tau-}. 
Since $g$ satisfies $\|g\|_*<\infty$ and \eqref{eq:inhom}, 
we can apply the linear theory established in 
\cite[Proposition 3.1]{WZ25} 
(see also \cite[Propositions 5.1, 7.1]{CDM20} and \cite[Proposition 7.1]{WZZ24})  
to \eqref{eq:Innerproblem_tau-}. 
Uniqueness is also a consequence of the linear theory, 
since both $\tilde \phi$ and $\tilde C$ depend linearly on $g$ 
(see \cite[Proposition 3.1]{WZ25}). 
The uniqueness together with rotational invariance of the equation 
in \eqref{eq:Innerproblem_tau-} implies the radial symmetry. 
Then  the lemma follows. 
\end{proof}

Let us prove Proposition \ref{pro:solvinner}. 

\begin{proof}[Proof of Proposition \ref{pro:solvinner}]
Let $\psi\in B_{\rm out}^\tau$ and let 
$\lambda=\lambda_\psi \in \Lambda^\tau$ be
determined by Proposition \ref{Proposition:SolvabilityforOurinner}. 
To apply Lemma \ref{lem:solvinn} 
with $g=\mathcal{G}[\psi,\lambda]\eta_{4R(t)} \mathbf{1}_{t \leq \tau}$, 
we check that 
$\| \mathcal{G}[\psi,\lambda]\eta_{4R(t)}\mathbf{1}_{t \leq \tau}\|_*<\infty$. 
By \eqref{eq:SSS} and $\psi\in B_{\rm out}^\tau$ (see \eqref{eq:spaceforouter-}), 
we note that 
\begin{align}
	&\label{eq:SSSuni}
	|\theta_A(x,t)| \lesssim |A| t^{-1}, \\
	&\label{eq:psiesBout222}
	|\psi(x,t)|\leq |A|t^{-1} R(t)^{-a_1}
	[\mathbf{1}_{|x| \le \sqrt{t}} 
	+ t|x|^{-2} \mathbf{1}_{|x| > \sqrt{t}}]
	\lesssim |A|t^{-1} R(t)^{-a_1}. 
\end{align}
From \eqref{eq:behabioroftau-}, \eqref{eq:defofH-}, 
\eqref{eq:leadingmu_0-}, \eqref{eq:SSSuni} and \eqref{eq:psiesBout222}, 
it follows that 
\[
\begin{aligned}
	\sigma(t) \langle y\rangle^{2+a} |\mathcal{G}| &\lesssim 
	t^{1- \frac{5A}{2}} \langle y\rangle^{2+a}
	\Big( |A| t^{-1+\frac{5A}{2}} \langle y \rangle^{-4}
	+|A| t^{-1+\frac{5A}{2}} \langle y \rangle^{-4}  \\
	&\quad 
	+|A| t^{-1+\frac{5A}{2}} R^{-a_1} \langle y \rangle^{-4}   
	 \Big) \\
	&\lesssim 
	2|A| \langle y \rangle^{-2+a}
	+|A| R^{-a_1} \langle y \rangle^{-2+a}    
	\lesssim |A|  
\end{aligned}
\]
for $1\ll t_0\leq t\leq \tau$, where $t_0$ depends on $|A|$. 
Thus, 
\[
	\| \mathcal{G}[\psi,\lambda] \eta_{4R(t)}\mathbf{1}_{t \leq \tau}\|_* \lesssim |A| <\infty, 
\]
and so we can apply Lemma \ref{lem:solvinn} to see that 
there exist $\phi_\psi\in B_{\rm in}^\infty$ and $C_\psi\in \mathbf{R}$ 
satisfying \eqref{eq:inhom} 
with $g=\mathcal{G}[\psi,\lambda]\eta_{4R(t)} \mathbf{1}_{t \leq \tau}$. 
Thus, the proof is complete.
\end{proof}

In addition, we show a kind of stability of $\phi_\psi$. 

\begin{lemma}\label{Lemma:stabilityofin}
Assume that $\tau > t_0 \gg1$. 
If $\psi_1\to \psi_2$ in $X^\tau_{\rm out}$ 
with $\psi_1, \psi_2 \in B_{\rm out}^\tau$, then 
$\phi_{\psi_1} \to \phi_{\psi_2}$ in $X^\tau_{\rm in}$.
\end{lemma}

\begin{proof}
Let $\psi_1, \psi_2 \in B_{\rm out}^\tau$.
We write $\lambda_i:=\lambda_{\psi_i}$
and $\sigma_i := \sigma_{\lambda_i}$ for $i=1,2$. 
By \eqref{eq:defofH-} and \eqref{eq:Innerproblem_tau-}, 
we observe that $\Phi_i(y,s):= \phi_{\psi_i}(y,t)$ satisfies
\begin{equation*}
\left\{ 
\begin{aligned}
	&\partial_s  \Phi_i - \Delta_y \Phi_i -2U(y) \Phi_i
	= \mathcal{G}[\psi_i, \lambda_i] 
    \eta_{4R(\sigma_i^{-1}(s))}
	\mathbf{1}_{s\leq \sigma_i(\tau)}
	&&y\in B_{16R(s)},\; s>s_0, \\
	&\Phi_i(y,s_0) = \tilde C_i  Z_0(y), 
	&&y \in B_{16R(s_0)}. 
\end{aligned}
\right.
\end{equation*}
Thus, by using Lemma \ref{lem:solvinn} 
with $\tilde \phi=\Phi_1-\Phi_2$,
\[
	\|\Phi_1-\Phi_2\|_{X^\tau_{\rm in}} 
	\lesssim \Big\|\mathcal{G}[\psi_1, \lambda_1] 
    \eta_{4R(\sigma_1^{-1}(s))}
	\mathbf{1}_{s\leq \sigma_1(\tau)} - \mathcal{G}[\psi_2, \lambda_2] 
    \eta_{4R(\sigma_2^{-1}(s))}
	\mathbf{1}_{s\leq \sigma_2(\tau)}\Big\|_*.
\]
By the continuity of $\lambda_\psi$ with respect to $\psi$ 
(see Lemma \ref{Lemma:stabilityofparameters}),
we see that the right hand side 
tends to $0$ as $\psi_1\to \psi_2$ in $X^\tau_{\rm out}$.
Thus, the proof is complete.
\end{proof}

\section{Outer problem}\label{sec:outer}
As shown in Sections \ref{sec:choice} and \ref{sec:inner}, 
if $\psi\in B_{\rm out}^\tau$ is given, 
then the modulation parameter $\lambda_\psi$ 
and the inner solution $\phi_\psi$ are uniquely determined. 
In this section, we construct a solution of the outer problem, 
and then we prove Theorem \ref{Theorem:A}. 
We first consider the existence of solutions of 
the outer problem. 

\begin{proposition}\label{Proposition:SolvabilityforOuter-}
Let $1/2<a_1<a<1$, $A\neq 0$ with $|A|\ll1$ and $\tau > t_0 \gg1$. 
For $\psi\in B_{\rm out}^\tau$, let 
$\lambda_\psi \in \Lambda^\tau$ and $\phi_\psi\in B_{\rm in}^\tau$ 
be determined by Propositions \ref{Proposition:SolvabilityforOurinner} 
and \ref{pro:solvinner}. 
Then there exists a solution 
$\psi\in B_{\rm out}^\tau$ of the outer problem 
\[
\left\{\begin{aligned}
	&\partial_t \psi - \Delta_x \psi 
	= \mathcal{H}[\phi_\psi, \psi, \lambda_\psi], 
	&&x\in \mathbf{R}^6, \; t_0<t<\tau, \\
	&\psi(\cdot,t_0)  = 0, 
	&&x\in \mathbf{R}^6, 
\end{aligned}
\right.
\]
where $\mathcal{H}$ is defined by \eqref{eq:defofG-}. 
\end{proposition}

As a preliminary, we clarify the sign of $u$ in a smaller inner region 
$|x| <\lambda^{{1/2}}t^{1/4}\sim t^{(5A/8) + (1/4)}$ 
(see \eqref{eq:leadingmu_0-}), 
where the inner region corresponding to the localization by $\eta(\tilde y)$ 
is $|x|\lesssim t^{1/2}$ (see \eqref{eq:Ansatzu}). 

\begin{lemma}\label{Lemma:signofu}
For $\phi \in B_{\rm in}^\tau$, $\psi\in B_{\rm out}^\tau$ and
$\lambda\in \Lambda^\tau$, 
$u$ of the form \eqref{eq:Ansatzu} satisfies
$u(x,t)>0$ for $|x| <\lambda^{{1/2}}t^{{1/4}}$
with $A\neq 0$ and $|A|\ll1$.
\end{lemma}

\begin{proof}
If $2\lambda R < |x| <  \lambda^{1/2}t^{1/4} (< t^{1/2})$, 
then $|y|<\lambda^{-1/2}t^{1/4}$, $\eta(\tilde y)=1$ and $\eta_R(y,t)=0$. 
By  $U(y) = ( 1 + |y|^2/24 )^{-2}$, 
Lemma \ref{thm:Psi_two_sided} and $\psi\in B_{\rm out}^\tau$, we see that 
\[
\begin{aligned}
u(x,t) 
&= \lambda(t)^{-2} U\left( y \right) 
	+ \theta_A(x,t)    +\psi(x,t)\\
&> \lambda(t)^{-2} \left( 1 + \frac{\lambda^{-1} t^\frac{1}{2}}{24}  \right)^{-2} - |A|t^{-1} 
- |A| t^{-1} R^{-a_1}\\
&\ge 2^{-1} 24^2 t^{-1} - |A|\left[1+ R^{-a_1}\right]t^{-1} >0 \\
\end{aligned}
\]
for $1\ll t_0\leq t\leq \tau$.
If $|x|\le 2\lambda R$, equivalently $|y|\le 2R$, 
then by \eqref{eq:behabioroftau-}, 
$\phi\in B_{\rm in}^\tau$ and $\psi\in B_{\rm out}^\tau$, we have 
\[
\begin{aligned}
u(x,t) 
&= \lambda(t)^{-2} U\left( y \right) 
	+ \theta_A(x,t)  +  \lambda(t)^{-2} \phi\left( y, t \right) \eta_R +\psi(x,t)\\
&> \lambda(t)^{-2} U(2R) - |A|t^{-1} - \lambda(t)^{-2}\sigma(t)^{-1}R^{7-a} \langle y\rangle^{-7}
- |A| t^{-1} R^{-a_1}\\
&> \lambda(t)^{-2} \left( 1 + \frac{R^2}{6}  \right)^{-2} - |A|t^{-1}
- C \lambda(t)^{-2}t^{-1+\frac{5A}{2}}R^{7-a} 
- |A| t^{-1} R^{-a_1}\\
&> C^{-1} t^{-\frac{5A}{2}}R^{-4}-C t^{-1}R^{7-a}  - |A|\left[1 + R^{-a_1}\right]t^{-1} >0 \\
\end{aligned}
\]
for $1\ll t_0\leq t\leq \tau$.
Thus, the proof is complete.
\end{proof}

Let us prove Proposition \ref{Proposition:SolvabilityforOuter-} 
by means of several lemmas and Schauder's fixed point theorem. 
It suffices to find a fixed point $\psi \in B_{\rm out}^\tau$ for 
\begin{equation}\label{eq:defofout}
\begin{aligned}
	&\psi = \mathcal{S}_{\rm out}^\tau[\psi],  \quad 
	\mathcal{S}_{\rm out}^\tau[\psi]
	:= \mathcal{T}_{\rm out}[\mathcal{H}[\phi_\psi, \psi, \lambda_\psi]],  \\
	&\mathcal{T}_{\rm out} [f] 
	:= \int_{t_0}^t \int_{\mathbf{R}^6}
	(4\pi (t-s))^{-3} e^{ -\frac{|x-y|^2}{4(t-s)}} 
	f(y,s) dyds \quad \mbox{ for }f=f(x,t), 
\end{aligned}
\end{equation}
where $\mathcal{H}$ is given by \eqref{eq:defofG-}. 
Since $\lambda_\psi$ and $\phi_\psi$ can be uniquely determined 
for each $\psi$, $\mathcal{S}_{\rm out}^\tau[\psi]$ is well-defined as a map.
For estimating $\mathcal{T}_{\rm out}[\mathcal{H}]$, 
we prepare the following two lemmas 
by straightforward modifications of \cite{WZZ24} 
in the form needed for our computations.
In the sequel,  if $t_1\le t_2$, 
then we interpret $\int_{t_2}^{t_1}(\cdots) ds$ as $0$. 
This convention will be used 
when estimating $\int_{t_0}^{t/2} (\cdots) ds$ 
in the case where $t/2<t_0$, 
since we always work for $t\geq t_0$. 
Moreover, we also regard $\nabla^0$ as the identity map. 

\begin{lemma}\label{Lemma:WZZ24A1}
Assume that 
$v(t)\ge0$, $0\le l_1(t)\le l_2(t)\le C_*\sqrt{t}$ and 
$C_l^{-1} l_i(t) \le l_i(s) \le C_ll_i(t)$ ($i=1,2$) 
for all $t/2\le s\le t$ and $t\ge t_0 \ge0$, where $C_*>0$ and $C_i\ge1$. 
Then for $k=0,1$, the following inequality holds for $b=0$ and $b=4$: 
\[
\begin{aligned}
	&\left| 
	\nabla^k_x \left( \mathcal{T}_{\rm out} \left[ v(t) |x|^{-b} 
	\mathbf{1}_{\{l_1(t) \le |x| \le l_2(t)\}}\right] \right) \right|
	\lesssim
	t^{-3-\frac{k}{2}} 
	e^{-\frac{|x|^2}{16t}}\int_{t_0}^{\frac{t}{2}} v(s) 
	l_2(s)^{6-b}  ds\\
	&+ \sup_{s \in [t/2,t]} v(s) 
	\left\{
	\begin{aligned}
	&\left\{
	\begin{aligned}
	&l_2(t)^{2-k} &&\mbox{ if } b=0 \\
	&l_1(t)^{-2-k} &&\mbox{ if } b=4 
	\end{aligned}
	\right.
	&&\mbox{ for } |x|\le l_1(t),  \\
	&\left\{
	\begin{aligned}
	&l_2(t)^{2-k} &&\mbox{ if } b=0\\
	&|x|^{-2-k} &&\mbox{ if } b=4 \\
	\end{aligned}
	\right.
	&&\mbox{ for } l_1(t)< |x| \le l_2(t), \\
	&|x|^{-4-k}e^{-\frac{|x|^2}{16t}} l_2(t)^{6-b} 
	&&\mbox{ for } |x|>l_2(t), 
	\end{aligned}
	\right.
\end{aligned}
\]
where `$\lesssim$' is independent of $t_0$ 
and $v(s)$ is regarded as $0$ when $s<t_0$.
\end{lemma}

\begin{proof}
We note that 
\begin{equation}\label{eq:kernelesti}
	\left| 
	\nabla_x \left( (4\pi (t-s))^{-3} e^{ -\frac{|x-y|^2}{4(t-s)}} \right)\right|
	\lesssim (t-s)^{-\frac{7}{2}} e^{-\frac{|x-y|^2}{5(t-s)}}. 
\end{equation}
Then, the desired inequality follows in the same way 
as \cite[Lemma A.1]{WZZ24}. 
\end{proof}

\begin{lemma}\label{Lemma:WZZ24A2}
Assume that $v(t)\ge0$ and  $t_0\ge0$. Then for $k=0,1$, 
the following inequality holds for $b=0$ and $b=4$: 
\[
\begin{aligned}
	&\left| \nabla^k_x \left( \mathcal{T}_{\rm out} 
	\left[ v(t) |x|^{-b} \mathbf{1}_{\{|x| \ge  \sqrt{t}\}}\right] \right)\right| \\
	&\lesssim 
	\left\{
	\begin{aligned}
	&t^{-\frac{b}{2}-\frac{k}{2}} \int_{t_0}^{\frac{t}{2}} v(s)  ds
	+ t^{1-\frac{k}{2}-\frac{b}{2}} \sup_{s \in [t/2,t]} v(s)
	&&\mbox{ for } |x|\le \sqrt{t},\\
	&t^{-\frac{k}{2}} |x|^{-b}\int_{t_0}^{\frac{t}{2}} v(s) ds 
	+ t^{1-\frac{k}{2}} |x|^{-b}\sup_{s\in[t/2,t]}v(s) 
	&&\mbox{ for } |x|>\sqrt{t}, \\
    \end{aligned}
    \right.
\end{aligned}
\]
where `$\lesssim$' is independent of $t_0$ and 
$v(s)$ is regarded as $0$ when $s<t_0$.
\end{lemma}

\begin{proof}
This lemma follows from  \cite[Lemma A.2]{WZZ24} 
with \eqref{eq:kernelesti}. 
\end{proof}

\begin{remark}\label{rem:k1}
The estimates for $k=1$ in Lemmas \ref{Lemma:WZZ24A1}, \ref{Lemma:WZZ24A2}, 
\ref{lem:cN}, \ref{lem:cE} and \ref{lem:tilcE}
are not needed for showing the self-mapping property of $\mathcal{S}_{\rm out}^\tau$ 
in Lemma \ref{lem:selfcS}, 
and are only used to derive the spatial equi-decay property 
in the proof of Proposition \ref{Proposition:SolvabilityforOuter-}. 
In particular, estimates for $k=1$ in Lemmas 
\ref{lem:cN}, \ref{lem:cE} and \ref{lem:tilcE}
need not be sharp. 
\end{remark}

In what follows, we write 
$\phi=\phi_\psi$ and $\lambda=\lambda_\psi$ 
when no confusion can arise. 
Among the estimates for $\mathcal{T}_{\rm out}[\mathcal{H}]$ 
in \eqref{eq:defofout}, 
the most technical part is to handle $\mathcal{N}=\mathcal{N}[\phi,\psi,\lambda]$, 
which is given by \eqref{eq:defofN} 
and arises from the nonlinearity. 
Let us start with the estimate of $\mathcal{T}_{\rm out} [\mathcal{N}]$. 

\begin{lemma}\label{lem:cN}
Let  $|A|\ll1$ and $\tau > t_0 \gg1$. 
Then $| \nabla^k ( \mathcal{T}_{\rm out} [ \mathcal{N} ])| 
\lesssim A^2 t^{-\frac{k}{2}} w_{\rm out}$ 
for $t_0\leq t\leq \tau$, $\psi\in B_{\rm out}^\tau$ and $k=0,1$. 
\end{lemma}

\begin{proof}
We divide $\mathcal{N}$ into the following three parts: 
\[
\begin{aligned}
\mathcal{N}(x,t) 
	&= \mathcal{N}_{\rm in}(x,t)
	+ \mathcal{N}_{\rm mid}(x,t) 
	+\mathcal{N}_{\rm out}(x,t)\\
	&:= \mathcal{N} \mathbf{1}_{|x|< \lambda^{\frac{1}{2}}t^\frac{1}{4} } 
	+ \mathcal{N}  \mathbf{1}_{\lambda^{\frac{1}{2}}t^\frac{1}{4} 
	\le|x|< \lambda^{\frac{1}{2}}t^\frac{1}{4} R^{\frac{1}{2}}}
	+ \mathcal{N}  \mathbf{1}_{\lambda^{\frac{1}{2}}
	t^\frac{1}{4} R^{\frac{1}{2}}\le|x| }. 
\end{aligned}
\]
First, we examine the most crucial part $\mathcal{N}_{\rm out}$. 
Since the inner region corresponding to the localization by $\eta_R(y,t)$ 
is $|x|\lesssim \lambda R \sim t^{5A/4} R$ 
(see \eqref{eq:leadingmu_0-}) 
and the growth of $R(t)$ is subpolynomial, 
we see that 
$\eta_R(y,t)=0$ for 
$|x|\geq \lambda^{1/2}t^{1/4} R^{1/2} 
\sim t^{(5A/2)+(1/4)} R^{1/2}$. 
Thus, \eqref{eq:defofN} together with \eqref{eq:Ansatzu} 
and the mean value formula yields 
\begin{equation}\label{eq:Nouteqcomp}
\begin{aligned}
	\mathcal{N}_{\rm out}
	&=
	\Big( |\lambda^{-2} U( y) 
	\eta( \tilde{y} ) 
	+ \theta_A(x,t) + \psi(x,t)| 
	\left( \lambda^{-2} U( y) 
	\eta\left( \tilde{y} \right) 
	+ \theta_A + \psi \right)\\
	&\quad -\lambda^{-4} 
	U(y)^2 \eta(\tilde{y})^2 
	- |\theta_A|\theta_A
	-2 \lambda^{-2} U(y) \eta(\tilde{y})
	(\theta_A+\psi) \Big) \mathbf{1}_{\lambda^{\frac{1}{2}}
	t^\frac{1}{4} R^{\frac{1}{2}}\le|x| }\\
	&= 
	\Big( 2(\lambda^{-2} U( y) 
	\eta( \tilde{y} ) + \psi) 
	\int_0^1 
	\left|\theta_A + \alpha \left( \lambda^{-2} U( y) 
	\eta\left( \tilde{y} \right)  + \psi \right)\right| d\alpha
	\\
	&\quad -\lambda^{-4} 
	U(y)^2 \eta(\tilde{y})^2 
	-2 \lambda^{-2} U(y) \eta(\tilde{y})
	(\theta_A+\psi) \Big) 
	\mathbf{1}_{\lambda^{\frac{1}{2}}
	t^\frac{1}{4} R^{\frac{1}{2}}\le|x| }, 
\end{aligned}
\end{equation}
and so by Young's inequality, 
\[
	|\mathcal{N}_{\rm out}|
	\lesssim 
	\left( \lambda^{-4} U( y)^2 \eta( \tilde{y} )^2  + \psi^2 
	+ \lambda^{-2} U( y) \eta( \tilde{y} ) |\theta_A| 
	+ |\psi| |\theta_A|  \right) 
	\mathbf{1}_{\lambda^{\frac{1}{2}}
	t^\frac{1}{4} R^{\frac{1}{2}}\le|x| }. 
\]

We recall that $U(y) = ( 1 + |y|^2/24 )^{-2}$, 
$\lambda(t)\sim t^{5A/4}$ by \eqref{eq:leadingmu_0-} and 
\begin{align}
	&\label{eq:psiestisp}
	|\psi(x,t)|\leq |A|t^{-1} R(t)^{-a_1}
	[\mathbf{1}_{|x| \le \sqrt{t}} 
	+ t|x|^{-2} \mathbf{1}_{|x| \ge \sqrt{t}}], \\
	&\notag 
	|\theta_A(x,t)|\lesssim |A| t^{-1}
	[ \mathbf{1}_{|x| \le \sqrt{t}}
	+ t|x|^{-2} \mathbf{1}_{|x| \ge \sqrt{t}}], 
\end{align}
by \eqref{eq:spaceforouter-} and \eqref{eq:SSS}, respectively.  
Then, we observe from Young's inequality that 
\[
\begin{aligned}
	|\mathcal{N}_{\rm out}|
	&\lesssim 
	\Big( t^{-5A}  \langle y\rangle^{-8} \eta( \tilde{y} )^2 
	+A^2t^{-2}R^{-2a_1}\left[\mathbf{1}_{|x| \le \sqrt{t}} 
	+ t|x|^{-2} \mathbf{1}_{|x| \ge \sqrt{t}}\right]^2 \\
	&\quad +t^{-\frac{5A}{2}}\langle y\rangle^{-4} \eta( \tilde{y} )
	\times |A|t^{-1} 
	\left[\mathbf{1}_{|x| \le \sqrt{t}} 
	+ t|x|^{-2} \mathbf{1}_{|x| \ge \sqrt{t}}\right] \\
	&\quad +|A|t^{-1} R(t)^{-a_1} \times |A|t^{-1} 
	\left[\mathbf{1}_{|x| \le \sqrt{t}} 
	+ t|x|^{-2} \mathbf{1}_{|x| \ge \sqrt{t}}\right]^2  \Big) 
	\mathbf{1}_{\lambda^{\frac{1}{2}}t^\frac{1}{4}R^\frac{1}{2}\le|x|}\\
	&\lesssim 
	t^{-5A}  \langle y\rangle^{-8} \eta( \tilde{y} )^2 (1+R^{2a_1}) 
	\mathbf{1}_{\lambda^{\frac{1}{2}}t^\frac{1}{4}R^\frac{1}{2}\le|x|} \\
	&\quad +A^2t^{-2}(R^{-2a_1}+R^{-a_1})\left[\mathbf{1}_{|x| \le \sqrt{t}} 
	+ t|x|^{-2} \mathbf{1}_{|x| \ge \sqrt{t}}\right]^2 
	\mathbf{1}_{\lambda^{\frac{1}{2}}t^\frac{1}{4}R^\frac{1}{2}\le|x|}. 
\end{aligned}
\]
If $|x|\geq \lambda^{{1/2}}t^{1/4}R^{1/2}$, 
then $|y|\geq \lambda^{-{1/2}}t^{1/4}R^{1/2} 
\sim t^{-(5A/8)+(1/4)}R^{1/2}$.  
This observation with $1/2<a_1<1$ 
and $R(t)\to\infty$ as $t\to\infty$ yield 
\begin{equation}\label{eq:contraction1-}
\begin{aligned}
	|\mathcal{N}_{\rm out}|
	&\lesssim 
	t^{-2} R^{-4+2a_1} \mathbf{1}_{|x| \le 2\sqrt{t}}
	\mathbf{1}_{\lambda^{\frac{1}{2}}t^\frac{1}{4}R^\frac{1}{2}\le|x|} \\
	&\quad +A^2t^{-2} R^{-a_1}\left[\mathbf{1}_{|x| \le \sqrt{t}} 
	+ t^2 |x|^{-4} \mathbf{1}_{|x| \ge \sqrt{t}}\right] 
	\mathbf{1}_{\lambda^{\frac{1}{2}}t^\frac{1}{4}R^\frac{1}{2}\le|x|} \\
	&\lesssim 
	A^2t^{-2} R^{-a_1} \mathbf{1}_{|x| \le 2\sqrt{t}} 
	+ A^2 R^{-a_1} |x|^{-4} \mathbf{1}_{|x| \ge \sqrt{t}}
\end{aligned}
\end{equation} 
for $1\ll t_0\leq t\leq \tau$, where $t_0$ depends on $A^2$. 

We estimate $\mathcal{T}_{\rm out}[\mathcal{N}_{\rm out}]$, 
where $\mathcal{T}_{\rm out}$ is given in \eqref{eq:defofout}. 
As for the first term 
$A^2t^{-2} R^{-a_1}\mathbf{1}_{|x| \le 2\sqrt{t}}$ in 
the right-hand side of \eqref{eq:contraction1-}, 
from Lemma \ref{Lemma:WZZ24A1} with $v(t)= A^2 t^{-2} R(t)^{-a_1}$, 
$b=0$, $l_1(t)=0$ and $l_2(t)=2\sqrt{t}$, it follows that 
\[
\begin{aligned}
	&\left| \nabla^k \left( \mathcal{T}_{\rm out} 
	[ A^2t^{-2} R^{-a_1}\mathbf{1}_{|x| \le 2\sqrt{t}} ] \right)\right|
	\lesssim 
	t^{-3-\frac{k}{2}} e^{-\frac{|x|^2}{16t}} 
	\int_{t_0}^\frac{t}{2} 
	A^2 s^{-2} R(s)^{-a_1} 
	(\sqrt{s})^6 ds  \\
	&+\sup_{s\in [t/2,t]} 
	A^2 s^{-2} R(s)^{-a_1} \left(  (\sqrt{t})^{2-k} 
	\mathbf{1}_{|x| \leq 2\sqrt{t}}
	+ |x|^{-4-k} e^{-\frac{|x|^2}{16t}} 
	(\sqrt{t})^6
	\mathbf{1}_{|x| \ge \sqrt{t}} \right) 
\end{aligned}
\]
for $k=0,1$. 
Let $0<\varepsilon<1$. 
By \eqref{eq:defR}, the function $s\mapsto s^\varepsilon R(s)^{-a_1}$ is 
increasing for $s\gg1$. 
Moreover, we observe that 
\begin{equation}\label{eq:kernelexp}
	e^{-\frac{|x|^2}{16t}} 
	= 
	e^{-\frac{|x|^2}{16t}}
	\mathbf{1}_{|x| \leq \sqrt{t}}
	+ e^{-\frac{|x|^2}{16t}} \mathbf{1}_{|x| \ge \sqrt{t}}
	\lesssim 
	\mathbf{1}_{|x| \leq \sqrt{t}}
	+ t |x|^{-2}  \mathbf{1}_{|x| \ge \sqrt{t}}
\end{equation}
and that $t^3 |x|^{-4-k}  \mathbf{1}_{|x| \ge \sqrt{t}} \leq 
t^{2-(k/2)} |x|^{-2}  \mathbf{1}_{|x| \ge \sqrt{t}}$. 
These together with the definition of 
$w_{\rm out}$ in \eqref{eq:spaceforouter-} yield  
\begin{equation}\label{eq:Noutint}
\begin{aligned}
	&\left| \nabla^k \left( \mathcal{T}_{\rm out} 
	[ A^2t^{-2} R^{-a_1}\mathbf{1}_{|x| \le 2\sqrt{t}} ] \right) \right| \\
	&\lesssim 
	A^2 t^{-3-\frac{k}{2}+\varepsilon}  R(t)^{-a_1}
	\int_{t_0}^\frac{t}{2} s^{1-\varepsilon} ds 
	\left( \mathbf{1}_{|x| \leq \sqrt{t}}
	+ t |x|^{-2}  \mathbf{1}_{|x| \ge \sqrt{t}} \right)
	 \\
	&\quad 
	+ A^2 t^\varepsilon R(t)^{-a_1}
	\sup_{s\in [t/2,t]} s^{-2-\varepsilon} \left(  t^{1-\frac{k}{2}} 
	\mathbf{1}_{|x| \leq 2\sqrt{t}}
	+ t^3 |x|^{-4-k}  \mathbf{1}_{|x| \ge \sqrt{t}} \right) \\
	&\lesssim 
	A^2 t^{-1-\frac{k}{2}} R(t)^{-a_1}
	\left( \mathbf{1}_{|x| \leq \sqrt{t}}
	+ t |x|^{-2}  \mathbf{1}_{|x| \ge \sqrt{t}} \right)
	= A^2 t^{-\frac{k}{2}} w_{\rm out}(x,t). 
\end{aligned}
\end{equation}

As for the second term 
$A^2 R^{-a_1} |x|^{-4} \mathbf{1}_{|x| \ge \sqrt{t}}$
in the right-hand side of \eqref{eq:contraction1-}, 
we apply Lemma \ref{Lemma:WZZ24A2} 
with $v(t)=A^2 R(t)^{-a_1}$ and $b=4$. Then, 
\[
\begin{aligned}
	&\left| \nabla^k \left( 
	\mathcal{T}_{\rm out} [ A^2 R^{-a_1} |x|^{-4} 
	\mathbf{1}_{|x| \ge \sqrt{t}}] \right) \right| \\
	&\lesssim 
	\left( t^{-2-\frac{k}{2}} \int_{t_0}^\frac{t}{2} 
	A^2 R(s)^{-a_1}  ds 
	+ t^{-1-\frac{k}{2}} \sup_{s\in [t/2,t]} A^2 R(s)^{-a_1} \right) 
	\mathbf{1}_{|x| \leq \sqrt{t}} \\
	&\quad 
	+ \left( t^{-\frac{k}{2}} |x|^{-4}
	\int_{t_0}^\frac{t}{2} A^2 R(s)^{-a_1} ds
	+ t^{1-\frac{k}{2}} |x|^{-4} \sup_{s\in [t/2,t]} A^2 R(s)^{-a_1}
	\right) \mathbf{1}_{|x| \geq \sqrt{t}}. 
\end{aligned}
\]
Again by the monotonicity of $s\mapsto s^\varepsilon R(s)^{-a_1}$, we see that 
\[
\begin{aligned}
	&\left| \nabla^k \left( \mathcal{T}_{\rm out} [ A^2 R^{-a_1} |x|^{-4} 
	\mathbf{1}_{|x| \ge \sqrt{t}}] \right) \right|  \\
	&\lesssim 
	A^2 \left( t^{-2-\frac{k}{2}+\varepsilon} R(t)^{-a_1}
	\int_{t_0}^\frac{t}{2} s^{-\varepsilon} ds 
	+ t^{-1-\frac{k}{2}+\varepsilon} R(t)^{-a_1} \sup_{s\in [t/s,t]} s^{-\varepsilon} \right) 
	\mathbf{1}_{|x| \leq \sqrt{t}} \\
	&\quad 
	+ A^2 |x|^{-4} \left( 
	t^{-\frac{k}{2} + \varepsilon} R(t)^{-a_1} \int_{t_0}^\frac{t}{2} s^{-\varepsilon} ds
	+ t^{1-\frac{k}{2}+\varepsilon} R(t)^{-a_1} \sup_{s\in [t/2,t]} s^{-\varepsilon} 
	\right) \mathbf{1}_{|x| \geq \sqrt{t}}  \\
	&\leq 
	A^2 t^{-1-\frac{k}{2}} R(t)^{-a_1}
	\left( \mathbf{1}_{|x| \leq \sqrt{t}}
	+ t |x|^{-2}  \mathbf{1}_{|x| \ge \sqrt{t}} \right)
	= A^2 t^{-\frac{k}{2}} w_{\rm out}(x,t). 
\end{aligned}
\]
Hence we obtain 
\begin{equation}\label{eq:ToutNoutesti}
	\left| \nabla^k \left( \mathcal{T}_{\rm out} [ \mathcal{N}_{\rm out} ] 
	\right) \right| 
	\lesssim A^2 t^{-\frac{k}{2}} w_{\rm out}
	\quad \mbox{ for }k=0,1. 
\end{equation}

We next consider $\mathcal{N}_{\rm in}$. 
It can be handled more simply than $\mathcal{N}_{\rm out}$.
In this region, $u>0$ by Lemma \ref{Lemma:signofu}. 
Then, from \eqref{eq:defofN} and Young's inequality, 
it follows that 
\[
\begin{aligned}
	|\mathcal{N}_{\rm in}|
	&=\Big| (\lambda^{-2} U \eta 
	+ \theta_A + \lambda^{-2} \phi \eta_R + \psi )^2 
	-\lambda^{-4} U^2 \eta^2 
	- |\theta_A|\theta_A
	\\
	&\quad -2 \lambda^{-2} U \eta
	(\theta_A+\psi+ \lambda^{-2}\phi\eta_R) \Big| 
	\mathbf{1}_{|x|< \lambda^{\frac{1}{2}}t^\frac{1}{4} }, \\
	&=\left| (\theta_A + \lambda^{-2} \phi \eta_R + \psi )^2 
	- |\theta_A|\theta_A  \right| 
	\mathbf{1}_{|x|< \lambda^{\frac{1}{2}}t^\frac{1}{4} } \\
	&\lesssim 
	\left( \lambda^{-4} \phi(y,t)^2 \eta_R(y,t)^2 
	+ \psi(x,t)^2  + \theta_A(x,t)^2  \right) 
	\mathbf{1}_{|x|< \lambda^{\frac{1}{2}}t^\frac{1}{4} }. 
\end{aligned}
\]
We recall from \eqref{eq:behabioroftau-}, 
\eqref{eq:spaceforinner-}, 
\eqref{eq:psiestisp}, \eqref{eq:SSS} and \eqref{eq:leadingmu_0-} that 
\begin{align}
	&\notag 
	|\phi(y,t)|\leq t^{-1+\frac{5A}{2}} R^{7-a} \langle y \rangle^{-7}, 
	\quad |\theta_A(x,t)|\lesssim |A| t^{-1}, \\
	&\label{eq:psiesBout}
	|\psi(x,t)|\leq |A|t^{-1} R(t)^{-a_1}
	[\mathbf{1}_{|x| \le \sqrt{t}} 
	+ t|x|^{-2} \mathbf{1}_{|x| > \sqrt{t}}]
	\lesssim |A|t^{-1} R(t)^{-a_1}. 
\end{align}
Since $\lambda(t)\sim t^{5A/4}$ (see \eqref{eq:leadingmu_0-}) and 
the localized region by $\eta_R(y,t)$ is $|y|\leq 2R$, 
equivalently $|x|\leq 2\lambda R$, we see that 
\[
	|\mathcal{N}_{\rm in}| \lesssim \Big( 
	t^{-2} R^{14-2a} \langle y \rangle^{-14} 
	\mathbf{1}_{|x| \le 2\lambda R} 
	+ A^2 t^{-2} R^{-2a_1} 
	+ A^2 t^{-2} 
	\Big) \mathbf{1}_{|x|< \lambda^{\frac{1}{2}}t^\frac{1}{4} }. 
\]
Then, from $\langle y \rangle\geq 1$, $|A|\ll 1$ 
and the subpolynomial growth of $R$, 
it follows that  
\begin{equation}\label{eq:Ninpw}
	|\mathcal{N}_{\rm in}| \lesssim 
	t^{-2} R^{14-2a} \mathbf{1}_{|x| \le 2 \lambda R} 
	+ A^2 t^{-2} \mathbf{1}_{|x|< \lambda^{\frac{1}{2}}t^\frac{1}{4} }
	\lesssim 
	A^2 t^{-2+\varepsilon} \mathbf{1}_{|x|< 2t^{\frac{5A}{8}+\frac{1}{4}}}
\end{equation}
for $1\ll t_0\leq t\leq \tau$, where $0<\varepsilon<1$ is a small constant. 
We observe that the region $|x|< 2t^{(5A/8)+(1/4)}$ 
is strictly smaller than $|x|\leq \sqrt{t}$ which appeared in \eqref{eq:Noutint}. 
Therefore, the smallness of $\varepsilon$ and 
the difference of the order between 
$t^{(5A/8)+(1/4)}$ and $\sqrt{t}$ for $t\geq t_0\gg1$ 
may imply that 
\begin{equation}\label{eq:ToutNinesti}
	\left| \nabla^k \left( \mathcal{T}_{\rm out} [ \mathcal{N}_{\rm in} ]
	\right) \right| 
	\ll A^2 t^{-\frac{k}{2}} w_{\rm out} 
	\quad \mbox{ for }k=0,1. 
\end{equation}

Let us check that \eqref{eq:ToutNinesti} indeed holds. 
By Lemma \ref{Lemma:WZZ24A1} 
with $v(t)=A^2 t^{-2+\varepsilon}$, $b=0$, $l_1(t)=0$ 
and $l_2(t)=2t^{(5A/8)+(1/4)}$, 
\[
\begin{aligned}
	&\left|\nabla^k \left( 
	\mathcal{T}_{\rm out}\left[\mathcal{N}_{\rm in}\right] \right)\right| 
	\lesssim 
	t^{-3-\frac{k}{2}} e^{-\frac{|x|^2}{16t}} 
	\int_{t_0}^\frac{t}{2} 
	A^2 s^{-2+\varepsilon} (s^{\frac{5A}{8}+\frac{1}{4}})^6 ds \\
	&\quad + A^2 t^{-2+\varepsilon} 
	\left( 
	(t^{\frac{5A}{8}+\frac{1}{4}})^{2-k} 
	\mathbf{1}_{|x|\leq 2t^{\frac{5A}{8}+\frac{1}{4}}} 
	+ |x|^{-4-k} e^{-\frac{|x|^2}{16t}} 
	(t^{\frac{5A}{8}+\frac{1}{4}})^6 
	\mathbf{1}_{|x|\geq 2t^{\frac{5A}{8}+\frac{1}{4}}} \right) \\
	&\lesssim 
	A^2 t^{-\frac{5}{2}+\frac{15A}{4}-\frac{k}{2}+\varepsilon}  e^{-\frac{|x|^2}{16t}} 
	+ A^2 t^{-\frac{3}{2}+\frac{5A}{4} -(\frac{5A}{8} + \frac{1}{4})k+\varepsilon}  
	\mathbf{1}_{|x|\leq 2t^{\frac{5A}{8}+\frac{1}{4}}} \\
	&\quad 
	+ A^2 t^{-\frac{1}{2}+\frac{15A}{4}+\varepsilon} 
	|x|^{-4-k} e^{-\frac{|x|^2}{16t}}  
	\mathbf{1}_{|x|\geq 2t^{\frac{5A}{8}+\frac{1}{4}}}. 
\end{aligned}
\]
By using \eqref{eq:kernelexp}, 
$\mathbf{1}_{|x|\leq 2t^{(5A/8)+(1/4)}} 
\leq \mathbf{1}_{|x|\leq \sqrt{t}}$, 
$|x|^{-4-k} \mathbf{1}_{|x|\geq 2t^{(5A/8) + (1/4) }} 
\lesssim t^{-(5A/2)-1}$ and 
$w_{\rm out}$ (see \eqref{eq:spaceforouter-}), 
\[
\begin{aligned}
	&\left|\nabla^k \left( 
	\mathcal{T}_{\rm out}\left[\mathcal{N}_{\rm in}\right] \right)\right| 
	\lesssim 
	A^2 t^{-\frac{5}{2}+\frac{15A}{4}-\frac{k}{2}+\varepsilon}  
	[\mathbf{1}_{|x| \leq \sqrt{t}}
	+ t |x|^{-2}  \mathbf{1}_{|x| \ge \sqrt{t}}] \\
	&\quad 
	+ A^2 t^{-\frac{3}{2}+\frac{5A}{4} -(\frac{5A}{8}+\frac{1}{4})k  +\varepsilon}  
	\mathbf{1}_{|x|\leq \sqrt{t}} 
	+ A^2 t^{-\frac{3}{2}+\frac{5A}{4} -(\frac{5A}{8}+\frac{1}{4})k
	+\varepsilon}  e^{-\frac{|x|^2}{16t}}  \\
	&\lesssim 
	A^2 t^{-\frac{3}{2}+\frac{15A}{4}-\frac{k}{2}+\varepsilon} R^{a_1} w_{\rm out} 
	+ 2 A^2 t^{-\frac{1}{2}+\frac{5A}{4}-(\frac{5A}{8}+\frac{1}{4})k
	+\varepsilon} R^{a_1} w_{\rm out} 
	\ll A^2 t^{-\frac{k}{2}} w_{\rm out}
\end{aligned}
\]
for $k=0,1$. Hence we obtain \eqref{eq:ToutNinesti}. 

Finally, we consider $\mathcal{N}_{\rm mid}$. 
It can be handled by the same method as 
for $\mathcal{N}_{\rm in}$ and $\mathcal{N}_{\rm out}$. 
By \eqref{eq:defofN},  $\eta_R(y,t)=0$ for 
$|x|\geq \lambda^{1/2} t^{1/4}$ 
and by the same applications of the mean value formula  
and Young's inequality as for $\mathcal{N}_{\rm out}$, 
\[
\begin{aligned}
	&|\mathcal{N}_{\rm mid}|
	=\Big| |\lambda^{-2} U \eta 
	+ \theta_A  + \psi|
	(\lambda^{-2} U \eta 
	+ \theta_A + \psi) \\
	&\quad 
	-\lambda^{-4} U^2 \eta^2 - |\theta_A|\theta_A 
	-2 \lambda^{-2} U \eta
	(\theta_A+\psi) \Big|
	\mathbf{1}_{\lambda^{\frac{1}{2}}t^\frac{1}{4} 
	\le|x|< \lambda^{\frac{1}{2}}t^\frac{1}{4} R^{\frac{1}{2}}} \\
	&\leq 
	\left( \lambda^{-4} U( y)^2 \eta( \tilde{y} )^2  + \psi^2 
	+ \lambda^{-2} U( y) \eta( \tilde{y} ) |\theta_A| 
	+ |\psi| |\theta_A|  \right) 
	\mathbf{1}_{\lambda^{\frac{1}{2}}t^\frac{1}{4} 
	\le|x|< \lambda^{\frac{1}{2}}t^\frac{1}{4} R^{\frac{1}{2}}}. 
\end{aligned}
\]
Then by $|A|\ll1$, we can compute that for $t\geq t_0\gg1$, 
\begin{equation}\label{eq:Nmidpw}
\begin{aligned}
	|\mathcal{N}_{\rm mid}|
	&\lesssim t^{-2}
	\mathbf{1}_{\lambda^{\frac{1}{2}}t^\frac{1}{4}
	\le|x|\le \lambda^{\frac{1}{2}}t^\frac{1}{4} R^\frac{1}{2}}
	\le  A^2 t^{-2+\varepsilon}
	\mathbf{1}_{|x|\le \lambda^{\frac{1}{2}}t^\frac{1}{4} R^\frac{1}{2}} \\
	&\lesssim 
	A^2 t^{-2+\varepsilon}
	\mathbf{1}_{|x|\le 2t^{\frac{5A}{8}+\frac{1}{4}} R^\frac{1}{2}} 
	\lesssim 
	A^2 t^{-2+\varepsilon}
	\mathbf{1}_{|x|\le t^{\frac{5A}{4}+\frac{1}{4}}}.
\end{aligned}
\end{equation}
Similarly to $\mathcal{N}_{\rm in}$, we obtain 
$|\nabla^k (\mathcal{T}_{\rm out} [ \mathcal{N}_{\rm mid} ])| \ll A^2 
t^{-{k/2}} w_{\rm out}$. 
This together with \eqref{eq:ToutNoutesti} and \eqref{eq:ToutNinesti} 
yields 
$| \nabla^k ( \mathcal{T}_{\rm out} [ \mathcal{N} ] )| 
\lesssim A^2 t^{-{k/2}} w_{\rm out}$. 
The proof is complete. 
\end{proof}

Next, we estimate $\mathcal{T}_{\rm out}[{\mathcal{E}}]$, 
where $\mathcal{T}_{\rm out}$ and ${\mathcal{E}}$ are given in 
\eqref{eq:defofout} and \eqref{eq:cEdef}. 

\begin{lemma}\label{lem:cE}
Let  $|A|\ll1$ and $\tau > t_0 \gg1$. 
Then $|\nabla^k (\mathcal{T}_{\rm out}[{\mathcal{E}}])| 
\ll A^2 w_{\rm out}$ 
for $t_0\le t \leq \tau$, $\psi\in B_{\rm out}^\tau$ and $k=0,1$. 
\end{lemma}

\begin{proof}
Since the proof of Lemma \ref{lem:cE}  
is easier than that of Lemma \ref{lem:cN}, 
we give fewer details.
We recall from 
\eqref{eq:behabioroftau-}, \eqref{eq:spaceforinner-}
and \eqref{eq:psiesBout} that 
\[
	|\phi| \lesssim t^{-1+\frac{5A}{2}} R^{7-a} \langle y \rangle^{-7}, 
	\quad 
	|\nabla \phi| \lesssim t^{-1+\frac{5A}{2}} R^{7-a} \langle y \rangle^{-8}, 
	\quad |\psi| \lesssim |A| t^{-1} R^{-a_1}. 
\]
In addition, from 
\eqref{eq:assmpfor_mu1_xi-}, \eqref{eq:leadingmu_000-}, \eqref{eq:leadingmu_0-} 
and $|\dot R/R|\ll t^{-1}$ 
by $R(t)=(\log(e+t))^2$, it follows that 
\[
	\lambda \sim t^\frac{5A}{4}, \quad 
	|\dot \lambda| \sim t^{-1+\frac{5A}{4}}, 
	\quad 
	\frac{|\partial_t(\lambda R)|}{\lambda R}
	\leq \frac{|\dot{\lambda}|}{\lambda} + \frac{|\dot{R}|}{R}
	\lesssim t^{-1}.
\]
Moreover, since $y={x/\lambda}$, we have 
\[
	U(y) = \left( 1 + \frac{|x|^2}{24\lambda^2} \right)^{-2}
	\lesssim \lambda^4 |x|^{-4} 
	\lesssim t^{5A} |x|^{-4}. 
\] 

From the above computations, it follows that 
\[
\begin{aligned}
	|{\mathcal{E}}|
	&\lesssim 
	\left( t^{-1-\frac{5A}{2}} R^{5-a} \langle y\rangle^{-7} 
	+ t^{-1-\frac{5A}{2}} R^{6-a} \langle y\rangle^{-8} 
	+ t^{-2} R^{7-a} \langle y\rangle^{-7} 
	\right) \mathbf{1}_{R\leq |y|\leq 2R}  \\
	&\quad 
	+ t^{-2} R^{7-a} \langle y\rangle^{-7} 
	\mathbf{1}_{|y|\leq 2 R}
	+ |A|t^{-1+\frac{5A}{2}} R^{-a_1} 
	|x|^{-4} 
	\mathbf{1}_{|x|\leq 2\sqrt{t}}
	\mathbf{1}_{|y|\geq R}. 
\end{aligned} 
\]
By $\langle y\rangle^{-1}\leq R^{-1}$ for $|y|\geq R$, 
$\langle y\rangle^{-1}\leq 1$ for $y\in\mathbf{R}^6$ and
the subpolynomial growth of $R$
we see that for $|A|\ll1$ and $t\geq t_0\gg1$, 
\[
\begin{aligned}
	|{\mathcal{E}}|
	&\lesssim 
	t^{-1-\frac{5A}{2}} R^{-2-a} 
	\mathbf{1}_{R\leq |y|\leq 2 R} 
	+  t^{-2} R^{7-a} 
	\mathbf{1}_{|y|\leq 2 R} \\
	&\quad 
	+ |A|t^{-1+\frac{5A}{2}} R^{-a_1} 
	|x|^{-4} 
	\mathbf{1}_{\lambda R \leq |x|\leq 2\sqrt{t}}. 
\end{aligned} 
\]
Thus, for some constant $C>1$, 
\begin{equation}\label{eq:cTcEout}
\begin{aligned}
	|{\mathcal{E}}|
	&\lesssim 
	t^{-1-\frac{5A}{2}} R^{-2-a} 
	\mathbf{1}_{|x|\leq 2\lambda  R} 
	+ |A|t^{-1+\frac{5A}{2}} R^{-a_1} |x|^{-4}
	\mathbf{1}_{\lambda R \leq |x|\leq 2\sqrt{t}}  \\
	&\lesssim 
	t^{-1-\frac{5A}{2}} R^{-2-a} 
	\mathbf{1}_{|x|\leq Ct^\frac{5A}{4}  R} 
	+ |A|t^{-1+\frac{5A}{2}} R^{-a_1} |x|^{-4}
	\mathbf{1}_{C^{-1} t^\frac{5A}{4} R \leq |x|\leq 2\sqrt{t}}. 
\end{aligned} 
\end{equation}

By Lemma \ref{Lemma:WZZ24A1} with 
$v(t)=t^{-1-(5A/2)} R(t)^{-2-a}$, $b=0$, $l_1(t)=0$ and 
$l_2(t)=Ct^{5A/4}R(t)$ 
and again by Lemma \ref{Lemma:WZZ24A1} with 
$v(t)=|A|t^{-1+(5A/2)} R(t)^{-a_1}$, $b=4$, 
$l_1(t)=C^{-1} t^{5A/4} R(t)$ and  $l_2(t)=2\sqrt{t}$, 
\[
\begin{aligned}
	&\left| \nabla^k \left(\mathcal{T}_{\rm out}[{\mathcal{E}}] \right) \right| 
	\lesssim
	t^{-3+5A-\frac{k}{2}} R^{4-a} e^{-\frac{|x|^2}{16t}} 
	+|A| t^{-2+\frac{5A}{2}-\frac{k}{2}} R^{-a_1}
	e^{-\frac{|x|^2}{16t}} \\
	&+ t^{-1-\frac{5A}{2}} R(t)^{-2-a}
	\Big( 
	t^{\frac{5A}{2}-\frac{5A}{4}k}  R^{2-k} 
	\mathbf{1}_{|x| \le Ct^\frac{5A}{4}R } 
	+ t^\frac{15A}{2} R^6 |x|^{-4-k}e^{-\frac{|x|^2}{16t}} 
	\mathbf{1}_{|x| \geq Ct^\frac{5A}{4}R } 
	\Big) \\
	&+ |A|t^{-1+\frac{5A}{2}} R(t)^{-a_1}  
	\Big( 
	t^{-\frac{5A}{2}-\frac{5A}{4}k} R^{-2-k} 
	\mathbf{1}_{|x|\le C^{-1} t^\frac{5A}{4} R}  \\
	&\quad 
	+ |x|^{-2-k} 
	\mathbf{1}_{C^{-1} t^\frac{5A}{4} R< |x| \le 2\sqrt{t}} 
	+ t |x|^{-4-k}e^{-\frac{|x|^2}{16t}}  
	\mathbf{1}_{|x| \geq 2\sqrt{t}}
	\Big)
\end{aligned}
\]
for $k=0,1$. 
Taking Remark \ref{rem:k1} into account, 
in the case $k=1$, 
we estimate the right-hand side by that for $k=0$. 
Then, we see that 
\[
\begin{aligned}
	&\left| \nabla^k \left(\mathcal{T}_{\rm out}[{\mathcal{E}}] \right) \right| 
	\lesssim
	t^{-3+5A} R^{4-a} e^{-\frac{|x|^2}{16t}} 
	+|A| t^{-2+\frac{5A}{2}} R^{-a_1}
	e^{-\frac{|x|^2}{16t}} \\
	&+ t^{-1-\frac{5A}{2}} R(t)^{-2-a}
	\Big( 
	t^{\frac{5A}{2}}  R^2 
	\mathbf{1}_{|x| \le Ct^\frac{5A}{4}R } 
	+ t^\frac{15A}{2} R^6 |x|^{-4}e^{-\frac{|x|^2}{16t}} 
	\mathbf{1}_{|x| \geq Ct^\frac{5A}{4}R } 
	\Big) \\
	&+ |A|t^{-1+\frac{5A}{2}} R(t)^{-a_1}  
	\Big( 
	t^{-\frac{5A}{2}} R^{-2} 
	\mathbf{1}_{|x|\le 2\sqrt{t}} 
	+ t |x|^{-4}e^{-\frac{|x|^2}{16t}}  
	\mathbf{1}_{|x| \geq 2\sqrt{t}}
	\Big)
\end{aligned}
\]
for $k=0,1$. 
By \eqref{eq:kernelexp} and $1/2<a_1<a<1$, we can continue to 
compute in a similar way to the derivation of \eqref{eq:Noutint} that 
for $k=0,1$, 
\[
\begin{aligned}
	\left| \nabla^k \left( \mathcal{T}_{\rm out}[{\mathcal{E}}] \right) \right| 
	&\lesssim 
	\left( t^{-1} R^{-a} + |A| t^{-1} R^{-a_1-2} \right)
	(\mathbf{1}_{|x| \leq \sqrt{t}}
	+ t |x|^{-2}  \mathbf{1}_{|x| \ge \sqrt{t}}) \\
	&\lesssim
	(|A|^{-1} R^{-(a-a_1)}+R^{-2}) w_{\rm out} 
	\ll A^2 w_{\rm out}
\end{aligned}
\]
for $t\geq t_0\gg1$, where $t_0$ depends on $|A|$. 
Hence we obtain the desired estimate. 
\end{proof}

Next, we estimate $\mathcal{T}_{\rm out}[\tilde {\mathcal{E}}]$, 
where $\mathcal{T}_{\rm out}$ and $\tilde {\mathcal{E}}=\tilde {\mathcal{E}}[\lambda]$ 
are given in \eqref{eq:defofout} and \eqref{eq:tilcEdef}. 

\begin{lemma}\label{lem:tilcE}
If $|A|\ll1$ and $\tau > t_0 \gg1$, then 
$| \nabla^k ( \mathcal{T}_{\rm out}[\tilde {\mathcal{E}}] ) |
\ll A^2 w_{\rm out}$ 
for $t_0\leq t\leq \tau$, $\psi\in B_{\rm out}^\tau$ and $k=0,1$. 
\end{lemma}

\begin{proof}
This lemma can be proved 
in the same way as Lemmas \ref{lem:cN} and \ref{lem:cE}.  
Thus, we only give an outline. 
Since $U(y)\leq \langle y\rangle^{-4}$, 
$|\nabla U(y)|\leq \langle y\rangle^{-5}$, 
$|\theta_A|\leq |A|t^{-1}\leq t^{-1}$ by \eqref{eq:SSS}, 
$\eta(\tilde y)(1-\eta_R(y,t))
\lesssim \mathbf{1}_{\lambda R\leq |x| \leq 2\sqrt{t}}$, 
$|\nabla \eta (\tilde{y})|, 
|\Delta \eta (\tilde{y})|, 
|\eta(\tilde y)^2 - \eta(\tilde y)|
\lesssim \mathbf{1}_{1\leq |\tilde y| \leq 2}$, 
one can see from \eqref{eq:tilcEdef} that 
\begin{equation}\label{eq:tilcEpw}
	|\tilde {\mathcal{E}}|\lesssim
	t^{-1+\frac{5A}{2}} |x|^{-4} 
	\mathbf{1}_{C^{-1}t^{\frac{5A}{4}} R \le |x| \le 2\sqrt{t}}. 
\end{equation}
By Lemma \ref{Lemma:WZZ24A1} with 
$v(t)=t^{-1+(5A/2)}$, $b=4$, 
$l_1(t)=C^{-1} t^{5A/4} R(t)$, $l_2(t)=2\sqrt{t}$, 
\[
\begin{aligned}
	&\left| \nabla^k \left( \mathcal{T}_{\rm out}[\tilde {\mathcal{E}}] \right) \right|
	\lesssim 
	t^{-2+\frac{5A}{2}-\frac{k}{2}}  e^{-\frac{|x|^2}{16t}}
	+ t^{-1+\frac{5A}{2}}
	\Big( 
	t^{-\frac{5A}{2}-\frac{5A}{4}k} R^{-2-k} 
	\mathbf{1}_{|x|\le C^{-1} t^\frac{5A}{4} R} \\
	&\quad + |x|^{-2-k} 
	\mathbf{1}_{C^{-1} t^\frac{5A}{4} R< |x| \le 2\sqrt{t}} 
	+ t |x|^{-4-k}e^{-\frac{|x|^2}{16t}}  
	\mathbf{1}_{ |x|>2\sqrt{t} } 
	\Big) \\
	&\lesssim 
	t^{-2+\frac{5A}{2}}  e^{-\frac{|x|^2}{16t}}
	+ t^{-1}
	\Big( 
	R^{-2} \mathbf{1}_{|x|\le 2\sqrt{t}} 
	+ t^{\frac{5A}{2}} |x|^{-2}e^{-\frac{|x|^2}{16t}}  
	\mathbf{1}_{ |x|>2\sqrt{t} } 
	\Big). 
\end{aligned}
\]
Therefore, for $k=0,1$, 
\[
	\left| \nabla^k \left( \mathcal{T}_{\rm out}[\tilde {\mathcal{E}}] \right) \right|
	\lesssim 
	t^{-1} R^{-2} 
	(\mathbf{1}_{|x| \leq \sqrt{t}}
	+ t |x|^{-2}  \mathbf{1}_{|x| \ge \sqrt{t}}) 
	\lesssim
	R^{-2+a_1} w_{\rm out} 
	\ll A^2 w_{\rm out}. 
\]
Then the lemma follows. 
\end{proof}

Let us next check the self-mapping property of $\mathcal{S}_{\rm out}^\tau$.  

\begin{lemma}\label{lem:selfcS}
Let $\mathcal{S}_{\rm out}^\tau$ be given as in \eqref{eq:defofout}. 
Then $\mathcal{S}_{\rm out}^\tau$ maps $B_{\rm out}^\tau$ into itself. 
\end{lemma}

\begin{proof}
Let $\psi\in B_{\rm out}^\tau$. 
Recall from \eqref{eq:spaceforinner-} and \eqref{eq:spaceforouter-} that 
$\psi\in B_{\rm out}^\tau$ and $\phi_\psi\in B_{\rm in}^\tau$ 
are radially symmetric. 
By the definition of $\mathcal{H}$ in \eqref{eq:defofH-}, 
the function $\mathcal{H}[\phi_\psi,\psi,\lambda_\psi](\cdot,t)$ 
is also radially symmetric for $t_0<t<\tau$. 
This together with the radial symmetry of the heat kernel 
implies that $\mathcal{S}_{\rm out}^\tau$ in \eqref{eq:defofout} 
is also radially symmetric.

Combining Lemmas \ref{lem:cN}, \ref{lem:cE} and \ref{lem:tilcE} 
and recalling $\mathcal{H}=\mathcal{N}+{\mathcal{E}}+\tilde {\mathcal{E}}$ from \eqref{eq:defofG-}, 
we see that if $|A|\ll1$ and $\tau > t_0 \gg1$, then 
\begin{equation}\label{eq:cSout}
	|\mathcal{S}_{\rm out}^\tau[\psi](x,t)| 
	\lesssim A^2 w_{\rm out}(x,t) 
	\quad \mbox{ for }x\in \mathbf{R}^6, \; t_0\leq t\leq \tau. 
\end{equation}
Thus, by the definition of 
$B_{\rm out}^\tau$ in \eqref{eq:spaceforouter-} 
with \eqref{eq:leadingmu_000-}, 
it suffices to show that 
\begin{equation}\label{eq:nabcSout}
	|\nabla\mathcal{S}_{\rm out}^\tau[\psi](x, t)|
	\lesssim A^2 t^{-1-\frac{5A}{4}} R^{-1-a_1} 
	\quad \mbox{ for }x\in \mathbf{R}^6, \; t_0\leq t\leq \tau. 
\end{equation}
We write $\tilde\psi:=\mathcal{S}_{\rm out}^\tau[\psi]$ for simplicity of notation. 
Then, $\tilde\psi$ satisfies
\[
	\tilde\psi_t - \Delta \tilde\psi 
	= \mathcal{H}[\phi_\psi, \psi, \lambda_\psi], \quad
	x\in \mathbf{R}^6,\; t>t_0. 
\]
Note that 
$|\tilde \psi| \lesssim A^2 w_{\rm out}$. 
From \eqref{eq:contraction1-}, \eqref{eq:Ninpw}, \eqref{eq:Nmidpw}, 
\eqref{eq:cTcEout} and \eqref{eq:tilcEpw}, 
it follows from $|A|\ll1$ and $\tau \geq t_0\gg1$ that 
\[
\begin{aligned}
	&|\mathcal{H}|\leq |\mathcal{N}|+|{\mathcal{E}}|+|\tilde {\mathcal{E}}| 
	\lesssim 
	A^2t^{-2} R^{-a_1} \mathbf{1}_{|x| \le \sqrt{t}} 
	+ A^2 R^{-a_1} |x|^{-4} \mathbf{1}_{|x| \ge \sqrt{t}} \\
	&\quad + A^2 t^{-2+\varepsilon} \mathbf{1}_{|x|< 2t^{\frac{5A}{8}+\frac{1}{4}}}
	+ A^2 t^{-2+\varepsilon}\mathbf{1}_{|x|\le t^{\frac{5A}{4}+\frac{1}{4}}} 
	+t^{-1-\frac{5A}{2}} R^{-2-a} \mathbf{1}_{|x|\leq Ct^\frac{5A}{4}  R} \\
	&\quad + |A|t^{-1+\frac{5A}{2}} R^{-a_1} |x|^{-4}
	\mathbf{1}_{C^{-1} t^\frac{5A}{4} R \leq |x|\leq 2\sqrt{t}}
	+ t^{-1+\frac{5A}{2}} |x|^{-4} 
	\mathbf{1}_{C^{-1}t^{\frac{5A}{4}} R \le |x| \le 2\sqrt{t}} \\
	&\lesssim 
	A^2 t^{-2+\varepsilon} 
	+t^{-1-\frac{5A}{2}} R^{-2-a} 
	+ |A|t^{-1-\frac{5A}{2}} R^{-a_1-4} 
	+ t^{-1-\frac{5A}{2}} R^{-4}
	\lesssim t^{-1-\frac{5A}{2}} R^{-2-a}. 
\end{aligned}
\]
By \cite[Lemma 4.1]{LWZZ24} with 
$\rho = \lambda_0 R^{1+{((a-a_1)/2)}}$ 
and by \eqref{eq:cSout} and 
$w_{\rm out}(x,t) \leq t^{-1} R^{-a_1}$ (see \eqref{eq:spaceforouter-}), 
we see from 
$|A|\leq1$ and $a_1<a$ that 
\[
\begin{aligned}
	\|\nabla \tilde \psi(\cdot,t)\|_{L^\infty(\mathbf{R}^6)} 
	&\lesssim 
	\left( \lambda_0 R^{1+ \frac{a-a_1}{2}} \right)^{-1} 
	A^2 t^{-1} R^{-a_1}  
	+ \left( \lambda_0 R^{1+ \frac{a-a_1}{2}} \right) 
	t^{-1-\frac{5A}{2}} R^{-2-a}   \\
	&\leq
	2t^{-1-\frac{5A}{4}} R^{-1- \frac{a+a_1}{2}} 
	\ll A^2 t^{-1-\frac{5A}{4}} R^{-1- a_1}
\end{aligned}
\]
for $\tau \geq t\geq t_0\gg1$, where $t_0$ depends on $A^2$. 
This shows \eqref{eq:nabcSout}. 
Therefore, if $|A|\ll1$, then 
$\|\mathcal{S}_{\rm out}^\tau[\psi]\|_{\rm out}\lesssim |A|\leq 1$ 
for $\psi\in B_{\rm out}^\tau$. 
The lemma follows. 
\end{proof}

Based on the above lemmas, 
we prove Proposition \ref{Proposition:SolvabilityforOuter-} 
by Schauder's fixed point theorem. 
We recall that the function spaces used in the proof below 
are given by \eqref{eq:spaceforinner-}, 
\eqref{eq:spaceforouter-} and \eqref{eq:scdef}. 

\begin{proof}[Proof of Proposition \ref{Proposition:SolvabilityforOuter-}]
In view of  Lemma \ref{lem:selfcS}, 
it suffices to show that 
$B_{\rm out}^\tau \ni \psi \mapsto 
\mathcal{S}_{\rm out}^\tau[\psi]\in B_{\rm out}^\tau$ is continuous and compact 
for applying Schauder's fixed point theorem. 
The continuity immediately follows from Lemmas \ref{Lemma:stabilityofparameters} and \ref{Lemma:stabilityofin}.

As for compactness, 
let $\{\psi^{(j)}\}_{j=1,2,\ldots} \subset B_{\rm out}^\tau$ 
and let $Q\subset \mathbf{R}^6\times [t_0,\tau]$ be a compact set. 
We write $\tilde\psi^{(j)}:=\mathcal{S}_{\rm out}^\tau[\psi^{(j)}]$. 
From Lemmas \ref{lem:cN}, \ref{lem:cE} and \ref{lem:tilcE} 
together with $w_{\rm out}(x,t) = t^{-1} R(t)^{-a_1}
[\mathbf{1}_{|x| \le \sqrt{t}} 
+ t|x|^{-2} \mathbf{1}_{|x| \ge \sqrt{t}}]$ in \eqref{eq:spaceforouter-}, 
it follows that 
\begin{equation}\label{eq:tilpsipw}
	|\nabla^k \tilde \psi^{(j)}(x,t)| 
	\lesssim w_{\rm out}(x,t)
	\leq 
	t^{-1} \mathbf{1}_{|x| \le \sqrt{t}} 
	+ |x|^{-2} \mathbf{1}_{|x| \ge \sqrt{t}} 
\end{equation}
for $x\in\mathbf{R}^6$, $t_0\leq t\leq\tau$ and $k=0,1$, 
where the right-hand side is independent of $j$. 
Thus, $\tilde \psi^{(j)}$ and $\nabla \tilde \psi^{(j)}$ 
are uniformly bounded on $Q$. 
Moreover, since $\mathcal{H}[\phi_\psi,\psi,\lambda_\psi]$ in \eqref{eq:defofout} 
is bounded up to $t=t_0$ by the proof of Lemma \ref{lem:selfcS}, 
we can check that 
$\tilde \psi^{(j)}$ and $\nabla \tilde \psi^{(j)}$ are H\"older continuous 
in $Q$ uniformly for $j$, in particular, 
they are equi-continuous in $Q$. 
The Ascoli-Arzel\'a theorem and the diagonal arguments guarantee 
the existence of $\tilde \psi\in C^1(\mathbf{R}^6\times[t_0,\tau])$ 
and a subsequence still denoted by $\tilde \psi^{(j)}$ such that 
for each compact set $Q\subset \mathbf{R}^6\times [t_0,\tau]$,  
\begin{equation}\label{eq:limtilpsi111}
	\tilde \psi^{(j)} \to \tilde \psi, \quad 
	\nabla \tilde \psi^{(j)} \to \nabla \tilde \psi \quad  
	\mbox{ uniformly in }Q\mbox{ as }j\to\infty. 
\end{equation}

We show the uniform convergence of $\tilde \psi^{(j)}$ 
and $\nabla \tilde \psi^{(j)}$ 
in $\mathbf{R}^6\times[t_0,\tau]$. 
Let $\varepsilon>0$. 
By \eqref{eq:tilpsipw}, we have 
$|\nabla^k \tilde \psi^{(j)}(x,t)| 
\lesssim |x|^{-2}$. 
Letting $j\to\infty$ gives 
$|\nabla^k \tilde \psi(x,t)|\lesssim |x|^{-2}$. 
Then there exists $\tilde R>0$ independent of $j$ such that 
\[
	|\nabla^k (\tilde \psi^{(j)} - \tilde \psi)| 
	\leq |\nabla^k \tilde \psi|+|\nabla^k \tilde \psi^{(j)}|
	\lesssim \varepsilon 
	\quad \mbox{ for }x\in \mathbf{R}^6\setminus B_{\tilde R}, \; t_0\leq t\leq \tau 
\]
for each $j$ and $k=0,1$. 
Then by \eqref{eq:limtilpsi111}, we see that 
\[
	\tilde \psi^{(j)} \to \tilde \psi, \quad 
	\nabla \tilde \psi^{(j)} \to \nabla \tilde \psi \quad  
	\mbox{ uniformly in }\mathbf{R}^6\times[t_0,\tau]\mbox{ as }j\to\infty. 
\]
Thus, the compactness of $\mathcal{S}_{\rm out}^\tau$ follows.
We can now apply Schauder's fixed point theorem 
to obtain the desired fixed point of $\psi=\mathcal{S}_{\rm out}^\tau[\psi]$. 
The proof is complete.
\end{proof}

\begin{remark}
The key to proving 
Proposition \ref{Proposition:SolvabilityforOuter-} is 
the cancellation of $|\theta_A|\theta_A$ in \eqref{eq:Nouteqcomp}. 
It yields the estimate \eqref{eq:contraction1-}. 
In particular, we obtain 
\begin{equation}\label{eq:Noutquad}
	|\mathcal{N}_{\rm out}[\phi_\psi, \psi, \lambda_\psi]| 
	\lesssim  \underbrace{A^2 t^{-2} R^{-a_1} 
	\left[\mathbf{1}_{|x| \le \sqrt{t}} 
	+ t^2|x|^{-4} \mathbf{1}_{|x| \ge \sqrt{t}}\right]
	}_{\text{dominant contribution from } |\theta_A||\psi|} 
	\ll \theta_A(x,t)^2.
\end{equation}
Thus, we can control the quadratic nonlinearity 
coming from the equation \eqref{eq:6Dt0}.
On the other hand, 
if we employ the solution of the linear heat equation 
\[
	\tilde{\theta}_A(x,t) 
	:= A \int_{\mathbf{R}^6} (4\pi t)^{-3} e^{-\frac{|x-y|^2}{4t}} (1+|y|)^{-2} dy
\]
instead of $\theta_A$ (see \cite[(3.4)]{WZ25} for instance), 
then such a cancellation does not occur. 
Consequently, the counterpart of \eqref{eq:Noutquad} becomes 
\[
	|\mathcal{N}_{\rm out}[\phi_\psi, \psi, \lambda_\psi]| 
	\lesssim 
	\underbrace{A^2 t^{-2} 
	\left[\mathbf{1}_{|x| \le \sqrt{t}} 
	+ t^2|x|^{-4} \mathbf{1}_{|x| \ge \sqrt{t}}\right]
	}_{\text{dominant contribution from } 
	|\tilde{\theta}_A|^2}\sim \tilde \theta_A(x,t)^2, 
\]
and so we cannot control the nonlinearity. 
Indeed, further computations show that 
$\mathcal{S}_{\rm out}^\tau$ 
is not necessarily a self-map of $B_{\rm out}^\tau$. 
Hence forward self-similar solutions 
are essential in our construction. 
\end{remark}

We are now in a position to prove Theorem \ref{Theorem:A}. 

\begin{proof}[Proof of Theorem \ref{Theorem:A}]
Proposition \ref{Proposition:SolvabilityforOuter-} yields 
a solution $(\lambda^{(\tau)}, \phi^{(\tau)}, \psi^{(\tau)})$ 
of the gluing system \eqref{eq:Innerproblem-} and \eqref{eq:Outerproblem-} 
on $(t_0,\tau)$ for each $\tau>t_0$. 
Let $\{\tau_i\}_{i=1}^\infty\subset (t_0, \infty)$ 
be an increasing sequence satisfying $\tau_i \to \infty$ 
as $i \to \infty$. 
Note that 
$(\lambda^{(\tau_i)}, \phi^{(\tau_i)}, \psi^{(\tau_i)})$ 
is a solution on $(t_0 ,t_0+j)$ for each $j$ with $i$ large. 
For $j=1$, since 
$\psi^{(\tau_i)} \in B_{\rm out}^{t_0+1}$ 
satisfies $\psi^{(\tau_i)}=\mathcal{S}_{\rm out}^{t_0+1}[\psi^{(\tau_i)}]$ 
and $\mathcal{S}_{\rm out}^{t_0+1}:B_{\rm out}^{t_0+1}\to B_{\rm out}^{t_0+1}$ 
is compact and continuous, 
there exists a subsequence 
$\{\tau_i^1\}_{i=1}^\infty \subset \{\tau_i\}_{i=1}^\infty$ 
such that $(\lambda^{(\tau_i^1)}, \phi^{(\tau_i^1)}, \psi^{(\tau_i^1)})$
converges to a solution 
$(\lambda^{(\infty)}, \phi^{(\infty)}, \psi^{(\infty)})$ 
of the gluing system on $(t_0,t_0+1)$. 
By repeating this argument for each $j\geq2$, 
we can extend $(\lambda^{(\infty)}, \phi^{(\infty)}, \psi^{(\infty)})$ 
to $(t_0, t_0+j)$ and we obtain a subsequence 
$\{\tau_i^j\}_{i=1}^\infty \subset \{\tau_i^{j-1}\}_{i=1}^\infty$ 
such that 
$(\lambda^{(\tau_i^j)}, \phi^{(\tau_i^j)}, \psi^{(\tau_i^j)})$
converges to $(\lambda^{(\infty)}, \phi^{(\infty)}, \psi^{(\infty)})$ 
for each $j$ as $i\to\infty$. 
By the diagonal argument, $(\lambda^{(\infty)}, \phi^{(\infty)}, \psi^{(\infty)})$ 
can be extended to $(t_0, \infty)$ 
as a global-in-time solution of the gluing system. 
Then, we obtain a solution $u$ of \eqref{eq:6Dt0} 
satisfying \eqref{eq:Ansatzu}. 

Let us consider the sign of $u$. 
Since $\psi(\cdot,t_0)\equiv0$, we have 
\[
\begin{aligned}
	u(x,t_0) 
	&= \lambda^{(\infty)}(t_0)^{-2} 
	U\left( \frac{x}{\lambda^{(\infty)}(t_0)}\right) 
	\eta\left( \frac{x}{\sqrt{t_0}}\right) + \theta_A(x,t_0)\\
	&\quad 
	+\lambda^{(\infty)}(t_0)^{-2} 
	\phi^{(\infty)}\left( \frac{x}{\lambda^{(\infty)}(t_0)}, t_0\right) 
	\eta\left( \frac{x}{\lambda^{(\infty)}(t_0)R(t_0)}\right). 
\end{aligned}
\]
From the above construction, 
it follows that $\phi^{(\infty)} \in B_{\rm in}^\infty$, 
where $B_{\rm in}^\infty$ is defined 
by replacing $t_0\leq t\leq \tau$ with $t_0\leq t<\infty$ 
in \eqref{eq:spaceforinner-}. 
If $A>0$ and $t_0\gg 1$, then 
\[
	u(x,t_0) \geq 
	C^{-1} \lambda^{(\infty)}(t_0)^{-2}
	\left[24\left\langle \frac{x}{\lambda(t_0)} \right\rangle^{-4} 
	- C \sigma(t_0)^{-1} R^{7-a} 
	\left\langle \frac{x}{\lambda(t_0)} \right\rangle^{-7} \right]>0. 
\]
By the maximum principle, $u$ is positive when $A>0$. 
As for $A<0$, we consider the region $ |x| > 2\sqrt{t}$.
From Lemma \ref{thm:Psi_two_sided}, it follows that 
\[
	u(x,t)
	= \theta_A(x,t) +\psi^{(\infty)}(x,t)\lesssim -|A|  |x|^{-2}
	+ |A| R^{-a_1}|x|^{-2} <0
\]
for  $t\gg1$.
This together with Lemma \ref{Lemma:signofu} implies 
that $u$ is sign-changing when $A<0$.
Therefore, after shifting $t=t_0$ to $t=0$, 
the resultant function is a global-in-time solution of \eqref{eq:SHE1}
satisfying the desired properties. 
The proof is complete.
\end{proof}

\appendix
\section{Overview of the gluing construction}\label{sec:blue}
In this section, we outline a blueprint for the desired solution and the strategy 
for its gluing construction.

\subsection{Blueprint}
We seek a solution of the form:
\[
	u(x,t) 
	= 
	\underbrace{ 
	\lambda^{-2}(t) U\left( y \right) 
	\eta\left( \tilde{y} \right) 
	+ \theta_A(x,t) 
	}_{\text{leading terms}} 
	+ \underbrace{ 
	\underbrace{ 
	\lambda^{-2}(t) \phi\left( y, t \right) 
	\eta_R(y,t) 
	}_{\text{an inner profile}}
	+\underbrace{ \psi(x,t) 
	}_{\text{an outer profile}}
	}_{\text{remainder terms}}
\]
for $(x,t) \in \mathbf{R}^6 \times (t_0, \infty)$, 
where $y=x/\lambda(t)$, $\tilde{y}=x/\sqrt{t}$ 
and unknown functions are $\lambda$, $\phi$ and $\psi$. 
Each of the components is as follows. 

\subsubsection*{Constants}
\begin{itemize}
\item
$A\neq 0$ is a constant with $|A|$ sufficiently small. 
\item
$1/2<a_1<a<1$ are fixed constants. 
\end{itemize}

\subsubsection*{Given functions}
\begin{itemize}
    \item 
    $\displaystyle{U(y) = \left( 1 + \frac{|y|^2}{24}  \right)^{-2}}$ 
    is the Aubin--Talenti bubble on $\mathbf{R}^6$. 
    \item 
    $\displaystyle{
    \theta_A(x,t) = (t+1)^{-1} \Theta_A\left(\frac{|x|}{\sqrt{t+1}}\right)}$ 
    is a radially symmetric self-similar solution 
    of $\partial_t u =\Delta u + |u|u$ 
    with $\Theta_A(0)=A$ and $\Theta_A'(0)=0$. 
    \item 
    $R(t):=(\log(e+t))^2$. 
    \item 
    $\eta\in C^\infty_0(\mathbf{R}^6)$ satisfies 
    $\eta(x) =1$ for $|x| \le 1$ and $\eta (x)=0$ for $|x| \ge 2$. 
    \item
    $\eta_R(y,t) = \eta(y/R(t))$. 
\end{itemize}

\subsubsection*{Modulation parameter} 
We construct $\lambda \in \Lambda^\tau$. 
\begin{itemize}
    \item
    $\displaystyle \Lambda^\tau= \left\{ 
	\lambda(t); \; 
	\lambda=\lambda_0+\mu>0, \;
	\mu  \in C^1([t_0,\tau]), \; 
	|\mu| \le \frac{\lambda_0}{9}, \; 
	|\dot \mu| \le \frac{|\dot{\lambda}_0|}{9}
	\right\}$. 
    \item 
	$\displaystyle \lambda_0(t) \sim t^\frac{5A}{4}$ 
	for $t>t_0\gg1$. 
    \item
    $\mu \in B_{\rm sc}^\tau 
	= \{f \in C([t_0, \tau];\mathbf{R}); \; \|f\|_{\rm sc} \le |A|\}$. 
	\item
	$\|f\|_{\rm sc} 
	= \sup_{t_0 \leq t\leq  \tau} \lambda_0(t)^{-1} |f(t)|$.
\end{itemize}

\subsubsection*{Inner profile}
We construct $\phi\in B_{\rm in}^\tau$. 
\begin{itemize}
    \item
    $B_{\rm in}^\tau = \left\{ \phi\in X_{\rm in}^\tau; 
	\; \|\phi\|_{\rm in} \le 1 \right\}$. 
    \item
    $X_{\rm in}^\tau = \left\{ \phi(y,t); \; 
	\begin{aligned} 
	&\mbox{$\phi(\cdot,s)$ is radially symmetric for $\sigma(t_0)\leq s\leq \sigma(\tau)$,} \\
	&\phi \in C^{1,0}\left( \bigcup_{\sigma(t_0)\leq s\leq \sigma(\tau)} B_{2R(s)} \times \{s\} \right),  
	\; \|\phi\|_{X_{\rm in}^\tau} <\infty
	\end{aligned}
	\right\}$. 
	\item 
	$\displaystyle 
	\|\phi\|_{\rm in}
	= \|\phi\|_{X_{\rm in}^\tau} 
	= \sup_{y \in B_{2R(s)}, \; \sigma(t_0)\leq s \leq \sigma(\tau)} 
	\frac{ \langle y\rangle|\nabla \phi(y,s)| + |\phi(y,s)|}{w_{\rm in}(y,s)}$. 
	\item
	$w_{\rm in}(y,s)= s^{-1} R(s)^{7-a} \langle y\rangle^{-7}$ with  
	$\langle y \rangle = \sqrt{1+|y|^2}$. 
    \item
    $\displaystyle \sigma(t) = \int_{t_0}^t \frac{ds}{\lambda(s)^2}
	+ \frac{t_0}{\lambda_0(t_0)^2}$ for $t_0\leq t\leq \tau$ 
	and 
	$\sigma(t) \sim t^{1-(5A/2)}$  for $t\gg1$
	in our case 
	$\lambda_0(t) \sim t^{5A/4}$. 
\end{itemize}

\subsubsection*{Outer profile}
We construct $\psi\in B_{\rm out}^\tau$. 
\begin{itemize}
    \item
    $\displaystyle 
    B_{\rm out}^\tau = \left\{ \psi\in X_{\rm out}^\tau ; \; 
	\begin{aligned}
	&|\psi| \leq |A| w_{\rm out}, \; 
	|\nabla\psi| \leq |A| t^{-1}\lambda_0(t)^{-1} R(t)^{-1-a_1}\\ 
	&\mbox{ for }x \in \mathbf{R}^6, \; t_0\leq t\leq \tau
	\end{aligned}
	\right\}$.  
	\item
	$X_{\rm out}^\tau = \left\{ \psi(x,t);\; 
	\begin{aligned}
	&\mbox{$\phi(\cdot,t)$ is radially symmetric for each $t_0\leq t\leq \tau$,} \\
	&\psi \in C^{1,0}(\mathbf{R}^6\times[t_0,\tau]), 
	\; \|\psi\|_{\rm out}<\infty
	\end{aligned}
	\right\}$. 
	\item
	$\displaystyle 
	\|\psi\|_{\rm out} 
	= \sup_{x \in \mathbf{R}^6, \; t_0\leq t\leq \tau } |\psi(x,t)| 
	+ \sup_{x \in \mathbf{R}^6, \; t_0\leq t\leq \tau } |\nabla \psi(x,t)|$. 
	\item
	$w_{\rm out}(x,t) = t^{-1} R(t)^{-a_1}
	[ \mathbf{1}_{|x| \le \sqrt{t}} 
	+ t|x|^{-2} \mathbf{1}_{|x| \ge \sqrt{t}} ]$. 
\end{itemize}

\subsection{Strategy}
The inner-outer gluing method is carried out as follows.

\subsubsection*{Step 1} 
We formulate the inner-outer gluing system. 
Fix $\tau\in(t_0,\infty)$ and $\psi \in B_{\rm out}^\tau$.

\subsubsection*{Step 2} 
Find $\lambda= \lambda_\psi^{(\tau)}\in \Lambda^\tau$ 
satisfying the orthogonality condition
\[
	\int_{\R^6}  \mathcal{G}[\psi, \lambda](y,t)Z(y)\eta_{4R(t)}(y) dy =0
	\quad \mbox{ for }t_0<t<\tau,
\]
where $\mathcal{G}$ and $Z$ are defined by \eqref{eq:defofH-} 
and \eqref{eq:kernels}, respectively. 

\subsubsection*{Step 3} 
Find $\phi=\phi_\psi^{(\tau)} \in B_{\rm in}^\tau$ solving the inner problem
\[
\lambda^2 \partial_t \phi - \Delta_y \phi 
	- 2 U(y) \phi 
	= \mathcal{G}[\psi, \lambda], 
    \quad
	y \in B_{4R(t)}, \; t_0<t<\tau. 
\]
The orthogonality condition in Step 2  
guarantees the existence of such $\phi$. 

\subsubsection*{Step 4}
Solve
\[
\left\{ 
\begin{aligned}
	&\partial_t \tilde{\psi} - \Delta_x \tilde{\psi} = \mathcal{H}[\phi, \psi, \lambda], 
	&&x\in \mathbf{R}^6, \; t_0<t<\tau, \\
	&\tilde{\psi}(\cdot,t_0)=0, &&x\in \mathbf{R}^6,
\end{aligned}
\right. 
\]
where $\mathcal{H}$ is defined by \eqref{eq:defofG-}.

\subsubsection*{Step 5}
Under appropriate conditions, 
$\lambda_\psi^{(\tau)}$, $\phi_\psi^{(\tau)}$ and $\tilde{\psi}$ 
can be uniquely determined by $\psi$.
This defines a map $\mathcal{S}_{\rm out}^\tau: \psi \mapsto \tilde{\psi}$. We show that
(i) $\mathcal{S}_{\rm out}^\tau$ maps 
$B_{\rm out}^\tau$ into itself; 
(ii) $\mathcal{S}_{\rm out}^\tau:B_{\rm out}^\tau \to B_{\rm out}^\tau$ 
is continuous; 
(iii) $\mathcal{S}_{\rm out}^\tau:B_{\rm out}^\tau \to B_{\rm out}^\tau$ 
is compact.

\subsubsection*{Step 6}
Using Schauder's fixed point theorem, 
we find a fixed point $\psi^{(\tau)}$ of $\mathcal{S}_{\rm out}^\tau$ 
and show that 
the triplet $(\lambda_\psi^{(\tau)}, \phi_\psi^{(\tau)}, \psi^{(\tau)})$ 
has a limit $(\lambda^{(\infty)}, \phi^{(\infty)}, \psi^{(\infty)})$ 
in a suitable sense as $\tau \to \infty$. 
Then, this is the desired solution.

\section*{Acknowledgments}
The authors are grateful to Professor Junichi Harada for introducing them 
to the problem considered in this paper and for valuable comments. 
The second author was supported in part 
by JSPS KAKENHI Grant Numbers 22KK0035, 23K12998 and 23K22402.
The third author was supported in part 
by JSPS KAKENHI Grant Numbers 22KK0035, 23K13005 and 25KJ0013.


\begin{thebibliography}{10}

\bibitem{AD25}
G. Ageno, M. del Pino, 
Infinite time blow-up for the three dimensional energy critical heat equation 
in bounded domains. 
Math. Ann. 391 (2025), no. 1, 1--94.

\bibitem{CGS89}
L. A. Caffarelli, B. Gidas, J. Spruck, 
Asymptotic symmetry and local behavior 
of semilinear elliptic equations with critical Sobolev growth.
Comm. Pure Appl. Math. 42 (1989), no. 3, 271--297.

\bibitem{CDM20}
C. Cort\'azar, M. del Pino, M. Musso, 
Green's function and infinite-time bubbling in the critical nonlinear 
heat equation. 
J. Eur. Math. Soc. (JEMS) 22 (2020), no. 1, 283--344.

\bibitem{dPMW19}
M. del Pino, M. Musso, J. Wei, 
Type II blow-up in the 5-dimensional energy critical heat equation. 
Acta Math. Sin. (Engl. Ser.) 35 (2019), no. 6, 1027--1042. 

\bibitem{DMW20}
M. del Pino, M. Musso, J. Wei, 
Infinite-time blow-up for the 3-dimensional energy-critical 
heat equation.
Anal. PDE 13 (2020), no. 1, 215--274.

\bibitem{dPMWZZpre}
M. del Pino, M. Musso, J. Wei, Q. Zhang, Y. Zhou, 
Type II Finite time blow-up 
for the three dimensional energy critical heat equation, 
preprint, arXiv:2002.05765.

\bibitem{FK12}
M. Fila, J. R. King, 
Grow up and slow decay in the critical Sobolev case. 
Netw. Heterog. Media 7 (2012), no. 4, 661--671.

\bibitem{GK03}
V. A. Galaktionov, J. R. King, 
Composite structure of global unbounded solutions 
of nonlinear heat equations with critical Sobolev exponents.
J. Differential Equations 189 (2003), no. 1, 199--233.

\bibitem{GV97}
V. A. Galaktionov, J. L. Vazquez, 
Continuation of blowup solutions of nonlinear heat equations 
in several space dimensions.
Comm. Pure Appl. Math. 50 (1997), no. 1, 1--67.

\bibitem{Ha20}
J. Harada, 
A type II blowup for the six dimensional energy critical heat equation. 
Ann. PDE 6 (2020), no. 2, Paper No. 13, 63 pp.

\bibitem{Ha25pre}
J. Harada, 
Oscillatory behavior of solutions to the critical Fujita equation in 6D, 
preprint, arXiv:2511.17891. 

\bibitem{HW82}
A. Haraux, F. B. Weissler, 
Nonuniqueness for a semilinear initial value problem.
Indiana Univ. Math. J. 31 (1982), no. 2, 167--189.

\bibitem{Ka96}
T. Kawanago, 
Asymptotic behavior of solutions of a semilinear heat equation 
with subcritical nonlinearity.
Ann. Inst. H. Poincar\'e{} C Anal. Non Lin\'eaire 13 (1996), no. 1, 1--15.

\bibitem{LWZZ24}
Z. Li, J. Wei, Q. Zhang, Y. Zhou, 
Long-time dynamics for the energy critical heat equation in $\mathbf{R}^5$. 
Nonlinear Anal. 247 (2024), Paper No. 113594, 15 pp.

\bibitem{Na06}
Y. Naito, 
An ODE approach to the multiplicity of self-similar solutions 
for semi-linear heat equations.
Proc. Roy. Soc. Edinburgh Sect. A 136 (2006), no. 4, 807--835.

\bibitem{Na20}
Y. Naito, 
Asymptotically self-similar behaviour of global solutions 
for semilinear heat equations with algebraically decaying initial data.
Proc. Roy. Soc. Edinburgh Sect. A 150 (2020), no. 2, 789--811.

\bibitem{PY03}
P. Pol\'a\v cik, E. Yanagida, 
On bounded and unbounded global solutions 
of a supercritical semilinear heat equation.
Math. Ann. 327 (2003), no. 4, 745--771.

\bibitem{PY14}
P. Pol\'a\v cik, E. Yanagida, 
Global unbounded solutions of the Fujita equation in the intermediate range.
Math. Ann. 360 (2014), no. 1-2, 255--266.

\bibitem{Qu08}
P. Quittner, 
The decay of global solutions of a semilinear heat equation.
Discrete Contin. Dyn. Syst. 21 (2008), no. 1, 307--318.

\bibitem{Sc12}
R. Schweyer, 
Type II blow-up for the four dimensional energy critical 
semi linear heat equation.
J. Funct. Anal. 263 (2012), no. 12, 3922--3983.

\bibitem{SW19}
M. Sobajima, Y. Wakasugi, 
Weighted energy estimates for wave equation 
with space-dependent damping term for slowly decaying initial data.
Commun. Contemp. Math. 21 (2019), no. 5, 1850035, 30 pp.

\bibitem{WZZ24}
J. Wei, Q. Zhang, Y. Zhou, 
On Fila-King conjecture in dimension four.
J. Differential Equations 398 (2024), 38--140.

\bibitem{WZ25}
J. Wei, Y. Zhou, 
Some global solutions to the energy-critical semilinear heat equation.
J. Elliptic Parabol. Equ. 11 (2025), no. 3, 2279--2301.

\end{thebibliography}
\end{document}